\documentclass[12pt,a4paper]{amsart}
\usepackage[top=1.15in, bottom=1.15in, left=1.17in, right=1.17in]{geometry}
\usepackage[T1]{fontenc}
\usepackage[utf8]{inputenc}
\usepackage{lmodern}
\usepackage{microtype}
\usepackage{ellipsis}
\usepackage{amsmath}
\usepackage{amssymb}
\usepackage[subrefformat=parens]{subcaption}
\usepackage[english]{babel}

\usepackage[breaklinks,colorlinks,cite color=blue,pdfusetitle]{hyperref}
\usepackage{amsthm}
\usepackage{cleveref}
\usepackage{thmtools}

\declaretheorem[numberwithin=section]{theorem}
\declaretheorem[numbered=no,name=Theorem]{theorem*}
\declaretheorem[sibling=theorem]{proposition}
\declaretheorem[sibling=theorem]{lemma}
\declaretheorem[sibling=theorem]{conjecture}
\declaretheorem[sibling=theorem]{corollary}
\declaretheorem[style=definition,sibling=theorem]{definition}
\declaretheorem[style=definition,sibling=theorem]{remark}

\declaretheorem[style=definition,sibling=theorem,qed=$\lozenge$]{example}

\usepackage{mathtools}
\usepackage{tikz}
\usetikzlibrary{cd}
\usepackage[backend=biber,maxbibnames=99]{biblatex}
\numberwithin{equation}{section}

\newcommand\OO{\mathcal O}
\newcommand\CC{\mathcal C}
\newcommand\II{\mathcal I}

\newcommand\HH{\mathcal H}
\newcommand\FF{\mathcal F}

\newcommand\SSc{\mathcal S}
\newcommand\TT{\mathcal T}

\newcommand\R{\mathbb R}

\newcommand\Z{\mathbb Z}
\newcommand\N{\mathbb N}

\newcommand\FI{\mathfrak I}

\newcommand\Fc{\mathfrak c}
\newcommand\Fd{\mathfrak d}

\newcommand\isoto{\xrightarrow{\raisebox{-2pt}[0pt][0pt]{$\mathmakebox[7pt]{\sim}$}}}

\DeclareRobustCommand\Dashv{\Relbar\joinrel\mathrel{|}}

\makeatletter
\newcommand*\bcdot{\mathpalette\bcdot@{.5}}
\newcommand*\bcdot@[2]{\mathbin{\vcenter{\hbox{\scalebox{#2}{$\m@th#1\bullet$}}}}}
\makeatother

\DeclareMathOperator\im{im}
\DeclareMathOperator\relint{relint}

\DeclareMathOperator\supp{supp}
\DeclareMathOperator\sig{sig}
\DeclareMathOperator\tc{tc}
\DeclareMathOperator\TC{TC}
\DeclareMathOperator*\argmax{argmax}
\DeclareMathOperator\dg{dg}
\DeclareMathOperator\rk{rk}

\tikzset{posetelm/.style={draw, fill, circle, minimum size=4pt, inner sep=0}}
\tikzset{posetelmm/.style={draw, thick, minimum size=5pt, inner sep=0}}
\tikzset{marking/.style={red}}
\tikzset{elmname/.style={blue}}
\tikzset{covrel/.style={thick}}
\tikzset{netedge/.style={thick,-latex}}

\makeatletter
\tikzset{
    dot diameter/.store in=\dot@diameter,
    dot diameter=1.2pt,
    dot spacing/.store in=\dot@spacing,
    dot spacing=6pt,
    dots/.style={
        line width=\dot@diameter,
        line cap=round,
        dash pattern=on 0pt off \dot@spacing,
        shorten <=2pt,
        shorten >=2pt
    }
}

\title[Marked poset polytopes]{A continuous family of marked poset polytopes}

\author[X. Fang, G. Fourier, J.-P. Litza, C. Pegel]{Xin Fang, Ghislain Fourier, Jan-Philipp Litza, Christoph Pegel}

\address{Mathematisches Institut, Universit\"at zu K\"oln, Germany}
\email{xinfang.math@gmail.com}
\address{Institute for Algebra, Number Theory, and Discrete Mathematics, Leibniz University Hannover, Germany}
\email{fourier@math.uni-hannover.de}
\address{Department of Mathematics, University of Bremen, 28359 Bremen, Germany}
\email{jplitza@math.uni-bremen.de}
\address{Institute for Algebra, Number Theory, and Discrete Mathematics, Leibniz University Hannover, Germany}
\email{pegel@math.uni-hannover.de}

\begin{document}
\begin{abstract}
For any marked poset we define a continuous family of polytopes, parametrized by a hypercube, generalizing the notions of marked order and marked chain polytopes. By providing transfer maps, we show that the vertices of the hypercube parametrize an Ehrhart equivalent family of lattice polytopes. The combinatorial type of the polytopes is constant when the parameters vary in the relative interior of each face of the hypercube. Moreover, with the help of a subdivision arising from a tropical hyperplane arrangement associated to the marked poset, we give an explicit description of the vertices of the polytope for generic parameters.
\end{abstract}

\maketitle

\section*{Introduction}

\subsection*{Motivation}
To any finite poset $(P,  \leq)$, two polytopes are associated by Stanley \cite{Sta86}, the order polytope and the chain polytope, both defined in the positive orthant of $\mathbb{R}^{P}$. The defining relations of the order polytope are the following: for $p,q\in P$, $p \leq q$, the coordinates $x_p$ and $x_q$ satisfy $x_p \leq x_q$. 
The cover relations of the poset give the facets of a cone and by restricting the cone to the $|P|$-dimensional unit cube $[0,1]^P$ one obtains the order polytope. The defining inequalities of the chain polytope are given by the maximal chains in $P$, i.e., a maximal chain $p_1 < \ldots < p_s$ in $P$ gives rise to the inequality $x_{p_1} + \ldots + x_{p_s} \leq 1$. 

Stanley showed that the number of lattice points of both polytopes is the same, one is parametrized by filters in $P$ and the other is parametrized by anti-chains in $P$. Further, he introduced a transfer map from the order polytope to the chain polytope, which is a lattice-preserving, piecewise-linear bijection. The existence of the transfer map implies Ehrhart equivalence of the two polytopes, i.e., they both have the same Ehrhart polynomial.

More than 60 years ago, Gelfand and Tsetlin \cite{GT50} introduced monomial bases for finite-dimensional irreducible representations of the algebraic group $\operatorname{GL}_n(\mathbb{C})$. For each such basis, there is a polytope whose lattice points parametrize the basis vectors,  nowadays known as the Gelfand--Tsetlin polytope. In 2005, Ernest Vinberg proposed different polytopes whose lattice points conjecturally parameterize another basis of each irreducible representation of $\operatorname{GL}_n(\mathbb{C})$. This conjecture was proved by Feigin, Littelmann and the second author in 2010 \cite{FFL11}. As observed by Ardila, Bliem and Salazar \cite{ABS11} shortly after, these two families of polytopes are related in a similar way as Stanley's order and chain polytopes. They introduced marked order and marked chain polytopes, i.e., one fixes a subset of poset elements consisting of at least all extremal elements, provides an integeral marking for them and obtains defining inequalities from covering relations between any elements and chains between marked elements for marked order and marked chain polytopes, respectively. They extended Stanley's transfer map to this setting, again implying Ehrhart equivalence of the marked order and marked chain polytopes.

Motivated by the recent work on linear degenerate flag varieties \cite{CFFFR17}, the first two authors introduced marked chain-order polytopes \cite{FF16}, which are mixtures of the two, i.e., for each order ideal of the poset, one imposes chain conditions on the coordinates in the order ideal, and order conditions on the coordinates in its complement. They proved that these marked chain-order polytopes form an Ehrhart equivalent family of lattice polytopes, containing the marked order and marked chain polytopes as extremal cases.

Stanley \cite{Sta86} showed that chain and order polytopes have the same number of vertices; Hibi, Li, Sahara and Shikama \cite{HLSS15} showed that this is also true for the edges of the polytopes. Hibi and Li \cite{HL16} conjectured that for a fixed poset, the $f$-vector of the order polytope is componentwise dominated by the $f$-vector of the chain polytope, and proved this conjecture for the facets. The second author extended the conjecture to marked order and marked chain polytopes. Note that in that case, the number of vertices may differ (depending on the poset). This conjecture is further formulated in \cite{FF16} in the setting of marked chain-order polytopes.

The Hibi-Li conjecture has the following geometric application. There exists a flat degeneration of the full flag variety to the toric variety associated to the Gelfand-Tsetlin polytope \cite{GL96}, and also a flat degeneration to the toric variety associated to its marked chain counterpart \cite{FFL17}. The Hibi-Li conjecture would allow a quantitative comparison of the two toric degenerations, for example a comparison of the number of torus fixed points.

\subsection*{Main results}
In the following we will explain a new approach towards the Hibi-Li conjecture and the construction of a large Ehrhart equivalent family of marked poset polytopes. Indeed, we define a continuous family of polyhedra parametrized by a hypercube. We start with a marked poset $(P, \lambda)$, where $P$ is a poset, $\lambda$ is an order-preserving integral marking of an induced subposet $P^*$ of $P$, and let $\tilde{P} = P \setminus P^*$. We define a marked poset polyhedron $\OO_t(P,\lambda)$ for each $t \in [0,1]^{\tilde{P}}$, such that for $t = \mathbf{0}$ we obtain the marked order polyhedron and for $t = \mathbf{1}$ we obtain the marked chain polyhedron. If $t$ is the characteristic function of an order ideal of $\tilde{P}$, then the corresponding polyhedron is the marked chain-order polyhedron. In this sense, this construction unifies all marked poset polytopes mentioned above. We note here, that if $t$ is in the interior of the hypercube, then the associated polyhedron is not necessarily a rational polyhedron.

The first main result of this paper is concerned with parameters $t\in\{0,1\}^{\tilde P}$.
\begin{theorem*}
If $t$ is a vertex of the hypercube, then $\mathcal{O}_t(P, \lambda)$ is a lattice polyhedron. Moreover if $P^*$ consists of at least all extremal elements of $P$, then $\mathcal{O}_t(P, \lambda)$ is integrally closed and Ehrhart equivalent to the marked order polytope associated to $(P,\lambda)$.
\end{theorem*}
This construction gives an Ehrhart equivalent family of lattice polytopes which is parametrized by vertices of the hypercube. Moving the parameter away from vertices of the parametrizing hypercube, the combinatorial type of the polytope varies.
However, we have the following result:
\begin{theorem*}
The combinatorial type of $\mathcal{O}_t(P, \lambda)$ is constant along relative interiors of faces of the parametrizing hypercube.
\end{theorem*}

When the parameter varies from the interior of a face in the hypercube to its boundary, we introduce the notion of continuous degenerations from one polyhedron $Q_0$ to another polyhedron $Q_1$ (\Cref{def:cont-deform}), in order to compare the face lattices of the corresponding polyhedra. As a corollary, we deduce that the $f$-vector of $Q_1$ is componentwise dominated by the $f$-vector of $Q_0$. We then apply the construction to marked poset polyhedra: the polyhedra parametrized by the relative interior of a face of the hypercube degenerate to the polyhedra parametrized by the boundary of the given face.

We turn back to the Hibi-Li conjecture. The face structure and especially the vertices of the marked order polytopes are described in \cite{JS14,Peg17}. On the other side, for marked chain polytopes, neither the face structure nor the vertices are known so far. Even in the case of the Gelfand--Tsetlin poset, the number of vertices is not known (\cite{FM17}). It turns out that, by changing the parameter $t$ to a generic one, the vertices can be located by using a subdivision arising from a tropical hyperplane arrangement associated to the marked poset.

\begin{theorem*}
The vertices of a generic marked poset polyhedron $\mathcal{O}_t(P, \lambda)$ with $t \in (0,1)^{\tilde{P}}$ are exactly the vertices in its tropical subdivision.
\end{theorem*}

We close with a first approach to the Hibi-Li conjecture for ranked posets. In this case, we can actually prove the conjecture for the facets of marked chain-order polytopes. Posets arising from representation theory are usually ranked, so we can cover all these cases.

The paper is organized as follows: In \Cref{sec:mpp} we define the universal family of marked poset polyhedra, recovering all the previously mentioned marked poset polytopes. In \Cref{sec:propmpp} we show that the polytopes in the family are images of the marked order polyhedron under a parametrized transfer map, using the transfer map we discuss some properties of this family. We introduce and study continuous degenerations of polyhedra and apply them to our setup in \Cref{sec:cont-deg}. We recall tropical hyperplane arrangements in \Cref{sec:tropical}, which is applied in \Cref{sec:generic-vertices} to study the vertices of the generic marked poset polyhedron. We consider the contraction to regular marked posets in \Cref{sec:poset-transforms} and study the facets in the situation of ranked posets in \Cref{sec:facetHL}.

\section{Marked poset polyhedra}\label{sec:mpp}
\subsection{Notations}
For a polyhedron $Q$ we denote by $\mathcal{F}(Q)$ the face lattice of $Q$ \cite{Zie95}. 

A \emph{partially ordered set} \((P, {\le})\) is a set \(P\) together with a reflexive, transitive and anti-symmetric relation \({\le}\).
We use the usual short term \emph{poset} and omit the relation \({\le}\) in notation when the considered partial order is clear from the context.
A finite poset is determined by its \emph{covering relations}:
we say \emph{\(p\) is covered by \(q\)} and write \(p\prec q\), if \(p<q\) and whenever \(p\le r\le q\) it follows that \(r=p\) or \(r=q\).

\subsection{Definitions}

We start with recalling the notion of a marked poset.

\begin{definition}\label{def:abs-mpp}
    Let \(P\) be a finite poset and \(\lambda\colon P^*\to\R\) be a real valued order-preserving map on an induced subposet \(P^*\subseteq P\).
    We say \((P,\lambda)\) is a \emph{marked poset} with \emph{marking} \(\lambda\), \emph{marked elements} \(P^*\) and denote by \({\tilde P} = P\setminus P^*\) the set of all \emph{unmarked elements}.
\end{definition}

We introduce the main object to be studied in this paper.

\begin{definition} \label{def:mpp}
    Let \((P,\lambda)\) be a marked poset such that \(P^*\) contains at least all minimal elements of \(P\).
    For \(\smash{t\in[0,1]^{\tilde P}}\) we define the \emph{marked poset polyhedron} \(\OO_t(P,\lambda)\) as the set of all \(x\in\R^P\) satisfying the following conditions:
    \begin{enumerate}
        \item for each \(a\in P^*\) an equation
                \(x_a = \lambda(a)\),
        \item for each saturated chain \(p_0\prec p_1 \prec p_2 \prec \cdots \prec p_r \prec p\) with \(p_0\in P^*\), \(p_i\in \tilde P\) for \(i\ge 1\), \(p\in P\) and \(r\ge 0\) an inequality
            \begin{equation} \label{eq:chain}
                (1-t_p) \left( 
                t_{p_1}\cdots t_{p_r} x_{p_0} + t_{p_2}\cdots t_{p_r} x_{p_1} + \cdots + t_{p_{r}}x_{p_{r-1}}+x_{p_{r}}
                \right) \le x_p,
            \end{equation}
            where \(t_p=0\) if \(p\in P^*\). Note that the inequalities for \(p\in P^*\) and \(r=0\) may be omitted, since they are consequences of \(\lambda\) being order-preserving.
    \end{enumerate}
    Since the coordinates in \(P^*\) are fixed, we sometimes consider the projection of \(\OO_t(P,\lambda)\) in \(\smash{\R^{\tilde P}}\) instead and we keep the same notation for the projection.
\end{definition}

For the rest of the paper we assume \((P,\lambda)\) to have at least all minimal elements marked, so that \Cref{def:mpp} always applies.

When not just the minimal but in fact all extremal elements of \(P\) are marked, the polyhedra \(\OO_t(P,\lambda)\) will all be bounded and hence referred to as \emph{marked poset polytopes}.
In this paper, whenever a terminology using the word ``polyhedron'' is introduced, the same term with ``polyhedron'' replaced by ``polytope'' is always implicitly defined for the case of all extremal elements of \((P,\lambda)\) being marked.

We will refer to the family of all \(\OO_t(P,\lambda)\) for \(\smash{t\in[0,1]^{\tilde P}}\) as the \emph{universal family of marked poset polyhedra} associated to the marked poset \((P,\lambda)\).
When at least one parameter \(t_p\) is in \((0,1)\), we call \(\OO_t(P,\lambda)\) an \emph{intermediate marked poset polyhedron} and when all \(t_p\) are in \((0,1)\) a \emph{generic marked poset polyhedron}.

\subsection{Examples: Marked Chain-Order Polyhedra}
\label{sec:mcop}

Marked poset polyhedra generalize the notion of order polytopes and chain polytopes \cite{Sta86}, marked order polytopes and marked chain polytopes \cite{ABS11}, as well as the marked chain-order polytopes in the sense of \cite{FF16}. We explain in this subsection how to recover these polytopes by specializing the parameter $t$. 

Consider the marked poset polyhedra $\OO_t(P,\lambda)$ for $t$ being a vertex of the hypercube $[0,1]^{\tilde P}$, i.e., $t\in\{0,1\}^{\tilde P}$.
Each such \(t\) uniquely corresponds to a partition \(\tilde P = C \sqcup O\) such that \(t\) is the characteristic function \(\chi_C\), i.e.,
\begin{equation*}
    t_p = \chi_C(p) =
    \begin{cases}
        1 & \text{for \(p\in C\),} \\
        0 & \text{for \(p\in O\).}
    \end{cases}
\end{equation*}
In this case, we denote the marked poset polyhedron \(\OO_t(P,\lambda)\) by \(\OO_{C,O}(P,\lambda)\) and refer to it as a \emph{marked chain-order polyhedron}.
The elements of \(C\) will be called \emph{chain elements} and the elements of \(O\) \emph{order elements}. We obtain the following description:

\begin{proposition} \label{prop:mcop}
    Given any partition \(\tilde P = C\sqcup O\), the marked chain-order polyhedron \(\OO_{C,O}(P,\lambda)\) is given by the following linear equations and inequalities:
    \begin{enumerate}
        \item for each \(a\in P^*\) an equation \(x_a=\lambda(a)\),
        \item for each chain element \(p\in C\) an inequality \(0\le x_p\),
        \item for each saturated chain \(a\prec p_1 \prec p_2 \cdots \prec p_r \prec b\) between elements \(a,b\in P^*\sqcup O\) with all \(p_i\in C\) and \(r\ge 0\) an inequality
            \begin{equation*}
                x_{p_1} + \cdots + x_{p_r} \le x_b - x_a.
            \end{equation*}
            As before, the case \(a,b\in P^*\) and \(r=0\) can be omitted. \qedhere
    \end{enumerate}
\end{proposition}

\begin{proof}
    Let \(t=\chi_C\in\smash{\{0,1\}^{\tilde P}}\) and consider a chain \(p_0\prec p_1\prec\cdots\prec p_r\prec p\) with \(p_0\in P^*\), \(p_i\in\tilde P\) for \(i\ge 1\), \(p\in P\) and \(r\ge 0\).
    This chain yields an inequality
    \begin{equation} \label{eq:chain-orig}
        (1-t_p) \left( 
        t_{p_1}\cdots t_{p_r} x_{p_0} + t_{p_2}\cdots t_{p_r} x_{p_1} + \cdots + x_{p_{r}}
        \right) \le x_p,
    \end{equation}
    where \(t_p=0\) if \(p\in P^*\).
    
    When \(p\in C\) we have \(t_p=1\) and \eqref{eq:chain-orig} becomes \(0\le x_p\).
    Since all minimal elements are marked, there is such a chain ending in \(p\) for each \(p\in C\) and hence we get \(0\le x_p\) for all \(p\in C\) this way.

    When \(p\in P^*\sqcup O\), we have \(t_p=0\) and \eqref{eq:chain-orig} reads
    \begin{equation*}
        t_{p_1}\cdots t_{p_r} x_{p_0} + t_{p_2}\cdots t_{p_r} x_{p_1} + \cdots + x_{p_{r}} \le x_p.
    \end{equation*}
    Since \(t_{p_i}=\chi_C(p_i)\), letting \(k\ge 0\) be maximal such that \(p_k\in P^*\sqcup O\), we obtain
    \begin{equation*}
       x_{p_k} + x_{p_{k+1}} + \cdots + x_{p_{r}} \le x_p,
    \end{equation*}
    which is equivalent to
    \begin{equation*}
        x_{p_{k+1}} + \cdots + x_{p_{r}} \le x_p - x_{p_k}.
    \end{equation*}

    Conversely, consider any chain \(a\prec p_1 \prec p_2 \cdots \prec p_r \prec b\) between elements \(a,b\in P^*\sqcup O\) with all \(p_i\in C\).
    If \(a\in P^*\), the chain is of the type to give a defining inequality as in \Cref{def:mpp} and we immediately get
    \begin{equation*}
        x_{p_1} + \cdots + x_{p_r} \le x_b - x_a.
    \end{equation*}
    If \(a\in O\), extend the chain downward to a marked element to obtain a chain
    \begin{equation*}
        q_0 \prec q_1 \prec \cdots \prec q_l \prec a \prec p_1 \prec \cdots \prec p_r \prec n.
    \end{equation*}
    Since \(a\) is the last element in the chain contained in \(P^*\sqcup O\), the above simplification for the inequality given by this chain yields
    \begin{equation*}
        x_{p_1} + \cdots + x_{p_r} \le x_b - x_a. \qedhere
    \end{equation*}
\end{proof}

\begin{remark}
    The term ``marked chain-order polytope'' is used differently in \cite{FF16}, where the definition only allows partitions \(\tilde P = C\sqcup O\) such that there is no pair \(p\in O\), \(q\in C\) with \(p<q\), i.e., \(C\) is an order ideal in \(\tilde P\).
    We call such a partition an \emph{admissible partition} and refer to \(\OO_{C,O}(P,\lambda)\) as an \emph{admissible marked chain-order polyhedron (polytope)}.
    In this paper, we allow arbitrary partitions for marked chain-order polyhedra instead of referring to this more general construction as ``layered marked chain-order polyhedra'' as suggested in \cite{FF16}.

When all \(t_p=0\) (resp.\ all \(t_p=1\)) for \(p\in\tilde{P}\) and \(P^*\) contains all extremal elements of \(P\), the marked poset polyhedron \(\OO_t(P,\lambda)\) coincides with the marked order polytope \(\OO(P,\lambda)\) (resp.\ the marked chain polytope \(\CC(P,\lambda)\)) introduced in \cite{ABS11}. The order polytopes (resp.\ chain polytopes) defined in \cite{Sta86} are special marked order polytopes (resp.\ marked chain polytopes), see \cite{ABS11}.
\end{remark}

\section{Properties of marked poset polyhedra}\label{sec:propmpp}

\subsection{Transfer Maps}

We will continue by proving that the polyhedra defined in \Cref{def:mpp} are in fact images of the marked order polyhedron under a parametrized transfer map.

\begin{theorem} \label{thm:univ-transfer}
    For \(t\in [0,1]^{\tilde{P}}\), the maps \(\varphi_t,\psi_t\colon \R^P\to\R^P\) defined by
    \begin{align*}
        \varphi_t(x)_p &=
        \begin{cases}
            x_p &\text{if \(p\in P^*\),} \\
            x_p - t_p \max_{q\prec p} x_q &\text{otherwise},
        \end{cases} \\
        \psi_t(y)_p &=
        \begin{cases}
            y_p &\text{if \(p\in P^*\),} \\
            y_p + t_p \max_{q\prec p} \psi_t(y)_q &\text{otherwise},
        \end{cases}
    \end{align*}
    are mutually inverse.
    Furthermore, \(\varphi_t\) restricts to a piecewise-linear bijection from \(\OO(P,\lambda)\) to \(\OO_t(P,\lambda)\).
\end{theorem}

Note that \(\psi_t\) is well-defined, since all minimal elements in \(P\) are marked.
Given \(t,t'\in\smash{[0,1]^{\tilde P}}\) the maps \(\psi_t\) and \(\varphi_{t'}\) compose to a piecewise-linear bijection 
\begin{equation*}
    \theta_{t,t'} = \varphi_{t'}\circ\psi_t \colon \OO_t(P,\lambda) \longrightarrow \OO_{t'}(P,\lambda),
\end{equation*}
such that \(\varphi_t=\theta_{\mathbf{0},t}\) and \(\psi_t=\theta_{t,\mathbf{0}}\).
We call the maps \(\theta_{t,t'}\) \emph{transfer maps}.

\begin{proof}[{Proof of \Cref{thm:univ-transfer}}]
    We start by showing that the maps are mutually inverse.
    For \(p\in P^*\) -- so in particular for \(p\) minimal in \(P\) -- we immediately obtain \(\psi_t(\varphi_t(x))_p=x_p\) and \(\varphi_t(\psi_t(y))_p=y_p\).
    Hence, let \(p\) be non-minimal, unmarked and assume by induction that \(\psi_t(\varphi_t(x))_q=x_q\) and \(\varphi_t(\psi_t(y))_q=y_q\) hold for all \(q<p\).
    We have
    \begin{align*}
        \psi_t(\varphi_t(x))_p &= \varphi_t(x)_p + t_p \max_{q\prec p} \psi_t(\varphi_t(x))_q = \varphi_t(x)_p + t_p \max_{q\prec p} x_q = x_p \\
        \intertext{and}
        \varphi_t(\psi_t(y))_p &= \psi_t(y)_p - t_p \max_{q\prec p} \psi_t(y)_q = y_p.
    \end{align*}
    Hence, the maps are mutually inverse. 

    We now show that \(\varphi_t\) maps \(\OO(P,\lambda)\) into \(\OO_t(P,\lambda)\).
    Let \(x\in\OO(P,\lambda)\) and \(y=\varphi_t(x)\).
    Given any saturated chain \(p_0\prec p_1 \prec p_2 \prec \cdots \prec p_r \prec p\) with \(p_0\in P^*\), \(p_i\in\tilde P\) for \(i\ge 1\) and \(p\in P\), we have \(y_{p_i} \le x_{p_i} - t_{p_i} x_{p_{i-1}}\) for \(i\ge 1\) by definition of \(\varphi_t\).
    Hence,
    \begin{equation} \label{eq:chaincheck}
        \arraycolsep=1.4pt
        \def\arraystretch{1.4}
        \begin{array}{lll}
            &(1-t_p) \big(
            t_{p_1}\cdots t_{p_r} y_{p_0} + t_{p_2}\cdots t_{p_r} y_{p_1} &+ \cdots + y_{p_{r}}
            \big) \\
            \le
            &(1-t_p) \big( 
            t_{p_1}\cdots t_{p_r} x_{p_0} + t_{p_2}\cdots t_{p_r} (x_{p_1}-t_{p_1} x_{p_0}) &+ \cdots + (x_{p_{r}}-t_{p_r} x_{p_{r-1}})
            \big) \\
            =
            & \multicolumn{2}{l}{(1-t_p) x_{p_{r}}
            \le (1-t_p) \max_{q\prec p} x_q
            \le x_p - t_p \max_{q\prec p} x_q = y_p.}
        \end{array}
    \end{equation}
    Thus, we have shown that \(y\in\OO_t(P,\lambda)\) as it satisfies \eqref{eq:chain} for all chains.

    Finally, we show that \(\psi_t\) maps \(\OO_t(P,\lambda)\) into \(\OO(P,\lambda)\).
    Let \(y\in \OO_t(P,\lambda)\) and \(x=\psi_t(y)\).
    Now consider any covering relation \(q\prec p\).
    If \(q\) is marked, the inequality \eqref{eq:chain} given by the chain \(q\prec p\) yields
    \begin{equation*}
        y_p \ge (1-t_p) x_q.
    \end{equation*}
    If \(q\) is not marked, set \(p_r\coloneqq q\) and inductively pick \(p_{i-1}\) such that \(x_{p_{i-1}} = \max_{q'\prec p_i} x_{q'}\) until ending up at a marked element \(p_0\).
    Inequality \eqref{eq:chain} given by the chain
    \begin{equation*}
        p_0\prec p_1 \prec \cdots \prec p_r=q \prec p.
    \end{equation*}
    still yields
    \begin{align*}
        y_p &\ge 
        (1-t_{p}) \left( 
        t_{p_1}\cdots t_{p_r} y_{p_0} + t_{p_2}\cdots t_{p_r} y_{p_1} + \cdots + y_{p_r}
        \right) \\
        &= 
        (1-t_{p}) \left( 
        t_{p_1}\cdots t_{p_r} x_{p_0} + t_{p_2}\cdots t_{p_r} (x_{p_1} - t_{p_1}x_{p_0}) + \cdots + t_{p_r} (x_{p_{r-1}}-t_{p_{r-1}} x_{p_{r-2}}) + y_q
        \right) \\
        &=
        (1-t_{p}) \left( 
        t_{p_r} x_{p_{r-1}} + y_q
        \right) = (1-t_{p}) \left( t_q \max_{q'\prec q} x_{q'} + y_q \right) = (1-t_p) x_q.
    \end{align*}
    Hence, if \(p\) is not marked, we have
    \begin{equation*}
        x_p = y_p + t_p \max_{q'\prec p} x_{q'} \ge y_p + t_p x_q \ge x_q.
    \end{equation*}
    If \(p\) is marked, \(t_p=0\) so \(x_p=y_p\ge x_q\).
    Thus, all defining conditions of \(\OO(P,\lambda)\) are satisfied.
\end{proof}

\begin{remark} \label{rem:explicit-psi}
    In contrast to the transfer maps defined in \cite[Theorem 3.2]{Sta86} and \cite[Theorem 3.4]{ABS11}, the inverse transfer map \(\psi_t\) in \Cref{thm:univ-transfer} is given using a recursion.
    Unfolding the recursion, we might as well express the inverse transfer map for \(p\in\tilde P\) in the closed form
    \begin{equation*}
        \psi_t(y)_p =
        \max_{\Fc}  \left(t_{p_1} \cdots t_{p_r} y_{p_0} + t_{p_2} \cdots t_{p_r} y_{p_1} + \cdots  + y_{p_r}\right),
    \end{equation*}
    where the maximum ranges over all saturated chains \(\Fc\colon p_0 \prec p_1 \prec \cdots \prec p_r\) with \(p_0\in P^*\), \(p_i\in\tilde P\) for \(i\ge 1\) and \(r\ge 0\) ending in \(p_r=p\).
\end{remark}

In examples it is often convenient to consider the projected polyhedron \(\OO_t(P,\lambda)\) in \(\smash{\R^{\tilde P}}\).
Accordingly we define projected transfer maps.
\begin{definition} \label{def:reduced-transfer}
    Denote by \(\pi_{\tilde P}\) the projection \(\R^P\to\smash{\R^{\tilde P}}\) and by \(\iota_\lambda\colon\smash{\R^{\tilde P}}\to\R^P\) the inclusion given by \(\iota(x)_a = \lambda(a)\) for all \(a\in P^*\).
    Define the projected transfer maps \(\varphi_t,\psi_t\colon\smash{\R^{\tilde P}}\to\smash{\R^{\tilde P}}\) by \(\pi_{\tilde P} \circ \varphi_t \circ \iota_\lambda\) and \(\pi_{\tilde P} \circ \psi_t \circ \iota_\lambda\), respectively.
\end{definition}

\subsection{Subdivision into products of simplices and simplicial cones}\label{sec:mop-JS-subdivision}

In \cite[Section~2.3]{JS14} the authors introduced a polyhedral subdivision of the marked order polytope into products of simplices. We briefly recall the analogous construction of this subdivision for possibly unbounded polyhedra to later transfer it to all marked poset polyhedra.

Assume that \( (P,\lambda)\) is a marked poset with at least all minimal elements marked. 
Let \(\II\colon\varnothing=I_0\subsetneq I_1\subsetneq \cdots\subsetneq I_r=P\) be a chain of order ideals in \(P\). For each \(p\in P\) we denote by \(i(\II,p)\) the smallest index \(k\) for which \(p\in I_k\). This chain is said to be compatible with the marking \(\lambda\), if for any \(a,b\in P^*\),
\begin{equation*}
i(\II,a)<i(\II,b)\ \ \text{if and only if}\ \ \lambda(a)<\lambda(b).
\end{equation*}
Every such chain of order ideals gives a partition \(\pi_\II\) of \(P\) into blocks by letting \(B_k=I_k\backslash I_{k-1}\) for \(k=1,\dots,r\). A block $B_k$ is called restricted, if \(P^*\cap B_k\neq\emptyset\). 

Since \(\II\) is compatible with \(\lambda\), we obtain a quotient marked poset \((P/\pi_\II,\lambda/\pi_\II)\) as in \cite[Proposition~3.9]{Peg17}, where \(P/\pi_\II\) is the induced poset of blocks in $\pi$,  \((P/\pi_\II)^*\) is the set of restricted blocks in $\pi_\II$ and $\lambda/\pi_\II\colon (P/\pi_\II)^*\to\mathbb{R}$ is the marking defined by $(\lambda/\pi_\II)(B_k)=\lambda(a)$ for a restricted block $B_k$ containing a marked element $a$.

The poset \(P/\pi_\II\) admits a linear extension \(P_\II\) by demanding \(B_k\leq B_\ell\) if and only if \(k\leq\ell\). The compatibility of \(\lambda\) with \(\II\) implies that the marking \(\lambda/\pi_\II\) yields a strict marking \(\lambda_\II\) on \(P_\II\). This gives surjections of marked posets
\begin{equation*}
(P,\lambda)\rightarrow (P/\pi_\II,\lambda/\pi_\II)\rightarrow (P_\II,\lambda_\II),
\end{equation*}
where the first map is the quotient map given in \cite[Proposition 3.9]{Peg17}, and the second map is a linear extension. By \cite[Proposition 3.11]{Peg17}, we obtain an inclusion of marked order polyhedra
\begin{equation*}
\OO(P_\II,\lambda_\II)\hookrightarrow \OO(P,\lambda).
\end{equation*}

We let $F_\II$ denote the image of \(\OO(P_\II,\lambda_\II)\) in \(\OO(P,\lambda)\), it consists of all points \(x\in\OO(P,\lambda)\) which are constant on each block \(B_k\) and weakly increasing along the linear order \(B_1,\dots,B_r\) of the blocks. By an argument analogous to \cite[Lemma 2.5]{JS14}, \(F_\II\) is a product of simplices and simplicial cones, which provides a subdivision of \(\OO(P,\lambda)\) into products of simplices and simplicial cones.

\subsection{Integrality, Integral Closure and Unimodular Equivalence}

When \((P,\lambda)\) comes with an integral marking, so \(\lambda(a)\in\Z\) for all \(a\in P^*\), the authors of \cite{ABS11} already showed that \(\OO(P,\lambda)\) and \(\CC(P,\lambda)\) are Ehrhart equivalent lattice polytopes when all the extremal elements in $P$ are marked.
In \cite{Fou16} a necessary and sufficient condition for \(\OO(P,\lambda)\) and \(\CC(P,\lambda)\) to be unimodular equivalent is given, which is generalized to admissible marked chain-order polytopes in \cite{FF16}.
In \textit{loc.cit}, it is also shown that all the admissible marked chain-order polytopes are integrally closed lattice polytopes.

In this section we assume integral markings containing all extremal elements throughout and generalize the results above to not necessarily admissible partitions.
Let us start by showing that under these assumptions all the marked chain-order polytopes are lattice polytopes.

\begin{proposition}
    For \(t\in\smash{\{0,1\}^{\tilde P}}\) the marked chain-order polytope \(\OO_t(P,\lambda)\) is a lattice polytope.
\end{proposition}

\begin{proof}
    When \(t\in\smash{\{0,1\}^{\tilde P}}\), the transfer map \(\varphi_t\colon \OO(P,\lambda)\to\OO_t(P,\lambda)\) is piecewise-unimodular.
    In particular, it maps lattice points to lattice points.
    When \(\im(\lambda)\subseteq\Z\), we know that the marked poset polytope \(\OO(P,\lambda)\) is a lattice polytope.
    We consider the subdivision  into products of simplices from \Cref{sec:mop-JS-subdivision}. As the image of the lattice polytope \(\OO(P_\II,\lambda_\II)\) under the lattice-preserving map \(\OO(P_\II,\lambda_\II)\to\OO(P,\lambda)\), each cell \(F_\II\) is a lattice polytope.
    Hence, all vertices in the subdivision of \(\OO(P,\lambda)\) are lattice points.
    Applying \(\varphi_t\) we obtain a subdivision of \(\OO_t(P,\lambda)\) with still all vertices being lattice points.
    Since the vertices of \(\OO_t(P,\lambda)\) have to appear as vertices in the subdivision, we conclude that \(\OO_t(P,\lambda)\) is a lattice polytope.
\end{proof}

\begin{corollary}
    The polytopes \(\OO_t(P,\lambda)\) for \(t\in\smash{\{0,1\}^{\tilde P}}\) are all Ehrhart equivalent.
\end{corollary}

\begin{proof}
    This is an immediate consequence of the transfer map being a piecewise-unimodular bijection when \(t\in\smash{\{0,1\}^{\tilde P}}\).
\end{proof}

\begin{proposition}
    The polytopes \(\OO_t(P,\lambda)\) for \(t\in\smash{\{0,1\}^{\tilde P}}\) are all integrally closed.
\end{proposition}

\begin{proof}
    We will reduce \(\OO_t(P,\lambda)\) being integrally closed to the fact that unimodular simplices are integrally closed.
    Since we have a polyhedral subdivision of \(\OO_t(P,\lambda)\) into cells \(\varphi_t(F_\II)\), it suffices to show that each cell is integrally closed.
    On the cell \(F_\II\), the transfer map \(\varphi_t\) is the restriction of a unimodular map \(\R^P\to\R^P\).
    Hence, it is enough to show that each cell \(F_\II\) in the subdivision of \(\OO(P,\lambda)\) is integrally closed.
    In fact, since \(F_\II\) is the image of \(\OO(P_\II,\lambda_\II)\) under a map that identifies the affine lattices spanned by the polytopes, it suffices to show that marked order polytopes associated to linear posets with integral markings are integrally closed.
    Since these are products of marked order polytopes associated to linear posets with integral markings only at the minimum and maximum, it is enough to show that these are integrally closed.
    However, these are just integral dilations of unimodular simplices.
\end{proof}

Having identified a family of Ehrhart equivalent integrally closed lattice polytopes, we now move on to the question of unimodular equivalences within this family.

Given a marked poset \((P,\lambda)\), we call an element \(p\in\tilde P\) a \emph{star element} if \(p\) is covered by at least two elements and there are at least two different saturated chains from a marked element to \(p\).
This notion has been used in \cite{FF16} to study unimodular equivalence of admissible marked chain-order polytopes.

A finer notion we will use in our discussion is that of a \emph{chain-order star element} with respect to a partition \(C\sqcup O\) of \(\tilde P\).

\begin{definition}
    Given a partition \(\tilde P=C\sqcup O\), an element \(q\in O\) is called a \emph{chain-order star element} if there are at least two different saturated chains \(s\prec q_1\prec \cdots \prec q_k \prec q\) with \(s\in P^*\sqcup O\) and all \(q_i\in C\) and there are at least two different saturated chains \(q \prec q_1\prec \cdots \prec q_k \prec s\) with \(s\in P^*\sqcup O\) and all \(q_i\in C\).
\end{definition}

Note that if \(C\sqcup O\) and \((C\sqcup\{q\})\sqcup(O\setminus\{q\})\) are admissible partitions for some \(q\in O\), i.e., \(C\) is an order ideal in \(\tilde P\) and \(q\) is minimal in \(O\), then \(q\) is an \((O,C)\)-star element if and only it is a star element in the sense of \cite{FF16}.

\begingroup
\makeatletter
\apptocmd{\theproposition}{\unless\ifx\protect\@unexpandable@protect\protect\footnotemark\fi}{}{}
\makeatother

\begin{proposition} \label{prop:unimodular-equiv}
    \footnotetext{In the proof of \Cref{prop:unimodular-equiv} we do not use integrality of the marking or \(P^*\) containing all extremal elements. Hence, the statement still holds when \(P^*\) only contains all minimal elements and the marking is not integral.}
    Let \(C\sqcup O\) be a partition of \(\tilde P\) and \(q\in O\) not a chain-order star element.
    Let \(O'=O\setminus\{q\}\) and \(C'=C\sqcup\{q\}\), then \(\OO_{C,O}(P,\lambda)\) and \(\OO_{C',O'}(P,\lambda)\) are unimodular equivalent.
\end{proposition}
\endgroup

\begin{proof}
    We have to consider the following two cases.
    \begin{enumerate}
        \item There is exactly one saturated chain \(s\prec q_1\prec \cdots \prec q_k \prec q\) with \(s\in P^*\sqcup O\) and all \(q_i\in C\).            
            Define the unimodular map \(\Psi\colon\R^P\to\R^P\) by letting
            \begin{equation*}
                \Psi(x)_p =
                \begin{cases}
                    x_q - x_s - \cdots - x_{q_k} &\text{if \(p=q\),} \\
                    x_p &\text{otherwise.}
                \end{cases}
            \end{equation*}
            We claim that \(\Psi(\OO_{C,O}(P,\lambda)) = \OO_{C',O'}(P,\lambda)\).
            The defining inequalities of \(\OO_{C,O}(P,\lambda)\) involving \(x_q\) are the following:
            \begin{enumerate}
                \item for each saturated chain \(q\prec p_1 \prec p_2 \cdots \prec p_r \prec b\) with \(b\in P^*\sqcup O\), \(p_i\in C\) and \(r\ge 0\) an inequality
                    \begin{equation*}
                        x_{p_1} + \cdots + x_{p_r} \le x_b - x_q,
                    \end{equation*}
                \item the inequality
                    \begin{equation*}
                        x_{q_1} + \cdots + x_{q_k} \le x_q - x_s.
                    \end{equation*}
            \end{enumerate}
            Applying \(\Psi\), these translate to
            \begin{enumerate}
                \item for each saturated chain \(q\prec p_1 \prec p_2 \cdots \prec p_r \prec b\) with \(b\in P^*\sqcup O\), \(p_i\in C\) and \(r\ge 0\) an inequality
                    \begin{equation}
                        x_{q_1} + \cdots + x_{q_k} + x_q + x_{p_1} + \cdots + x_{p_r} \le x_b - x_s, \label{eq:start-in-q}
                    \end{equation}
                \item the inequality
                    \begin{equation}
                        0 \le x_q. \label{eq:end-in-q}
                    \end{equation}
            \end{enumerate}
            These are exactly the defining properties of \(\OO_{C',O'}(P,\lambda)\) involving \(x_q\):
            the saturated chains \(a\prec p_1 \prec p_2 \cdots \prec p_r \prec b\) with \(a,b\in P^*\sqcup O'\) and all \(p_i\in C'\) involving \(q\) at index \(k\), must have \(a=s\) and \(p_i=q_i\) for \(i\le k\), so they yield the inequalities in \eqref{eq:start-in-q}.
            The inequality in \eqref{eq:end-in-q} is what we get from \(q\in C'\).
        \item There is exactly one saturated chain \(q \prec q_1\prec \cdots \prec q_k \prec s\) with \(s\in P^*\sqcup O\) and all \(q_i\in C\).
            An analogous argument as above shows that in this case the map \(\Psi\colon\R^P\to\R^P\) defined by
            \begin{equation*}
                \Psi(x)_p =
                \begin{cases}
                    x_s - x_q - \cdots - x_{q_k} &\text{if \(p=q\),} \\
                    x_p &\text{otherwise}
                \end{cases}
            \end{equation*}
            yields an unimodular equivalence of \(\OO_{C,O}(P,\lambda)\) and \(\OO_{C',O'}(P,\lambda)\).
            In this case every chain involving \(q\) that is relevant for \(\OO_{C',O'}(P,\lambda)\) must end in \(\cdots \prec q \prec q_1\prec \cdots \prec q_k \prec s\). \qedhere
    \end{enumerate}
\end{proof}

\subsection{Combinatorial Types}

Having studied the marked chain-order polyhedra obtained for \(t\in\smash{\{0,1\}^{\tilde P}}\), we will now consider intermediate and generic parameters.
In this section we show that the combinatorial type of \(\OO_t(P,\lambda)\) stays constant when \(t\) varies inside the relative interior of a face of the parametrizing hypercube \(\smash{[0,1]^{\tilde P}}\).

The idea is to translate whether a defining inequality of \(\OO_t(P,\lambda)\) is satisfied for some \(\varphi_t(x)\) with equality into a condition on \(x\) depending only on the face the parameter \(t\) is contained in.
The key ingredient will be a relation on \(P\) depending on \(x\in\OO(P,\lambda)\).

\begin{definition} \label{def:maximizing-rel}
    Given \(x\in \OO(P,\lambda)\) let \(\dashv_x\) be the relation on \(P\) given by
    \begin{equation*}
        q \dashv_x p \quad\Longleftrightarrow\quad
        q\prec p \enskip\text{ and }\enskip x_q = \max_{q'\prec p} x_{q'}.
    \end{equation*}
\end{definition}

\begin{proposition} \label{prop:ineq-pullback}
    Let \(x\in\OO(P,\lambda)\). Given a saturated chain \(p_0\prec p_1\prec \cdots \prec p_r \prec p\) with \(p_0\in P^*\), \(p_i\in \tilde P\) for \(i\ge 1\) and \(p\in P\), the corresponding defining inequality \eqref{eq:chain} is satisfied with equality by \(\varphi_t(x)\) if and only if one of the following is true:
    \begin{enumerate}
        \item \(t_p=1\) and \(x_p = \max_{q\prec p} x_q\),
        \item \(t_p<1\) and \(x_p = x_{p_r}\) as well as
            \begin{equation*}
                p_{k-1} \dashv_x p_k \dashv_x \cdots \dashv_x p_r,
            \end{equation*}
            where \(k\ge 1\) is the smallest index such that \(t_{p_i}>0\) for all \(i\ge k\).
    \end{enumerate}
\end{proposition}

\begin{proof}
    Let \(y=\varphi_t(x)\in\OO_t(P,\lambda)\).
    When \(t_p=1\), the inequality \eqref{eq:chain} for \(y\) reads \(0\le y_p\) which is equivalent to
    \begin{equation*}
        \max_{q\prec p} x_q \le x_p.
    \end{equation*}

    When \(t_p<1\) we may simplify \eqref{eq:chain} to
    \begin{equation} \label{eq:shortened}
        (1-t_p) \left( 
        t_{p_k}\cdots t_{p_r} y_{p_{k-1}} + \cdots + y_{p_{r}}
        \right) \le y_p,
    \end{equation}
    where \(k\ge 1\) is the smallest index such that \(t_{p_i}>0\) for all \(i\ge k\). 
    The coefficients on the left hand side of \eqref{eq:shortened} are all strictly positive and inspecting the estimation in \eqref{eq:chaincheck} yields equality if and only if
    \begin{align*}
        x_{p_{i-1}} &= \max_{q\prec p_i} x_q \quad\text{for \(i\ge k\) \quad and} \\
        x_{p_r} &= x_p. \qedhere
    \end{align*}
\end{proof}

Since the conditions of \Cref{prop:ineq-pullback} only depend on each \(t_p\) being \(0\), \(1\) or in between, we obtain the following corollary.

\begin{corollary} \label{cor:combtype}
    The combinatorial type of \(\OO_t(P,\lambda)\) is constant along relative interiors of faces of the parametrizing hypercube \(\smash{[0,1]^{\tilde P}}\).
\end{corollary}

Furthermore, some of the \(t_p\) do not affect the combinatorial type at all:

\begin{proposition} \label{prop:irrelevant-p}
    The combinatorial type of \(\OO_t(P,\lambda)\) does not depend on \(t_p\) for \(p\in\tilde P\) such that there is a (unique) chain \(p_1\prec p_2 \prec \cdots \prec p_r \prec p_{r+1}=p\), where all \(p_i\in\tilde P\), \(p_1\) covers only marked elements and \(p_i\) is the only element covered by \(p_{i+1}\) for \(i=1,\dots,r\).
\end{proposition}
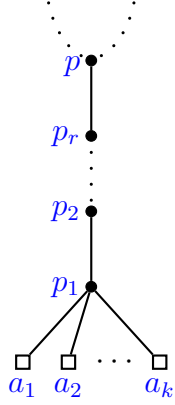
\begin{figure}
    \centering
    \begin{tikzpicture}[baseline={([yshift=-.5ex]current bounding box.center)},xscale=.6,yscale=1]
        \path (-1.5,0) node[posetelmm] (a1) {} node[below=2pt,elmname] {\(a_1\)};
        \path (-.5,0) node[posetelmm] (a2) {} node[below=2pt,elmname] {\(a_2\)};
        \path (.5,0) node {\(\cdots\)};
        \path (1.5,0) node[posetelmm] (ak) {} node[below=2pt,elmname] {\(a_k\)};
        \path (0,1) node[posetelm] (p1) {} node[left,elmname] {\(p_1\)};
        \path (0,2) node[posetelm] (p2) {} node[left,elmname] {\(p_2\)};
        \path (0,3) node[posetelm] (pr) {} node[left,elmname] {\(p_r\)};
        \path (0,4) node[posetelm] (p) {} node[left,elmname,yshift=-2pt] {\(p\)};
        \draw[covrel]
             (a1) -- (p1) -- (p2) (pr) -- (p)
             (a2) -- (p1)
             (ak) -- (p1);
        \draw[dots] (p2) -- (pr);
        \draw[dots] (p) to[bend left] +(-1,1);
        \draw[dots] (p) to[bend right] +(1,1);
    \end{tikzpicture}
    \caption[Condition on \(p\in P\) in \Cref{prop:irrelevant-p}]{The condition on \(p\in P\) in \Cref{prop:irrelevant-p}.}
    \label{fig:irrelevant-p}
\end{figure}
The condition on \(p\) in \Cref{prop:irrelevant-p} is equivalent the subposet of all elements below \(p\) being of the form depicted in \Cref{fig:irrelevant-p}.

\begin{proof}
    Let \(t,t'\in\smash{[0,1]^{\tilde P}}\) such that \(t_q=t'_q\) for \(q\neq p\) and consider the transfer map \(\theta_{t,t'}\).
    For \(y\in\OO_t(P,\lambda)\) and \(q\neq p\) we have
    \begin{equation*}
        \theta_{t,t'}(y)_q = \varphi_{t'}(\psi_t(y))_q = \psi_t(y)_q - t'_q \max_{q'\prec q} \psi_t(y)_{q'} = \psi_t(y)_q - t_q \max_{q'\prec q} \psi_t(y)_{q'} = y_q.
    \end{equation*}
    For the \(p\)-coordinate note that the given condition means walking down from \(p\) in the Hasse diagram of \((P,\lambda)\) we are forced to walk along \(p \succ p_r \cdots \succ p_1\) and \(p_1\) covers only marked elements.
    Hence, we have
    \begin{align*}
        \theta_{t,t'}(y)_p
        &= \psi_t(y)_p - t'_p \max_{q\prec p} \psi_t(y)_q
        = \psi_t(y)_p - t'_p \psi_t(y)_{p_r} \\
        &= y_p + t_p y_{p_r} + t_p t_{p_r} y_{p_{r-1}} + t_p t_{p_{r-1}} t_{p_r} y_{p_{r-2}} \cdots + t_p t_{p_1} t_{p_2} \cdots t_{p_r} \max_{a\prec p_1} \lambda(a)\\
        &\qquad - t'_p\left(y_{p_r} + t_{p_r} y_{p_{r-1}} + t_{p_{r-1}} t_{p_r} y_{p_{r-2}} \cdots + t_{p_1} t_{p_2} \cdots t_{p_r} \max_{a\prec p_1} \lambda(a)\right) \\
        &= y_p + (t_p - t'_p)\left(y_{p_r} + t_{p_r} y_{p_{r-1}} + t_{p_{r-1}} t_{p_r} y_{p_{r-2}} \cdots + t_{p_1} t_{p_2} \cdots t_{p_r} \max_{a\prec p_1} \lambda(a)\right).
    \end{align*}
    We conclude that \(\theta_{t,t'}\) restricts to an affine isomorphism \(\OO_t(P,\lambda)\isoto\OO_{t'}(P,\lambda)\).
\end{proof}

\begin{corollary}
Let \(k\) be the number of elements in \(\tilde P\) not satisfying the condition in \Cref{prop:irrelevant-p}, there are at most \(3^k\) different combinatorial types of marked poset polyhedra associated to a marked poset \((P,\lambda)\). \qed
\end{corollary}

\section{Continuous Degenerations}
\label{sec:cont-deg}

By \Cref{cor:combtype}, the combinatorial type of \(\OO_t(P,\lambda)\) is constant along the relative interiors of the faces of the hypercube \(\smash{[0,1]^{\tilde P}}\).
Assume we are looking at some \(\OO_t(P,\lambda)\) with \(t_p\in(0,1)\) for a fixed \(p\).
Continuously changing \(t_p\) to \(0\) or \(1\), the combinatorial type of the polyhedron stays constant until it possibly jumps, when reaching \(0\) or \(1\), respectively.
This motivates to think of the two polyhedra for \(t_p=0\) and \(t_p=1\) as continuous degenerations of the polyhedron for any \(t_p\in(0,1)\).

In this section we formally introduce a concept of continuous degenerations of polyhedra to then apply it to marked poset polyhedra.

\subsection{Continuous Degenerations of Polyhedra}

We start by defining continuous deformations of polyhedra, mimicking the situation in the universal family.

\begin{definition} \label{def:cont-deform}
    Given two polyhedra \(Q_0\) and \(Q_1\) in \(\R^n\), a \emph{continuous deformation} from \(Q_0\) to \(Q_1\) consists of the following data:
    \begin{enumerate}
        \item A continuous map \(\rho\colon Q_0\times [0,1]\to\R^n\), such that each \(\rho_t=\rho(-,t)\) is an embedding, \(\rho_0\) is the identical embedding of \(Q_0\) and the image of \(\rho_1\) is \(Q_1\).
        \item Finitely many continuous functions \(f^1,f^2,\dots,f^k\colon \R^n\times [0,1]\to\R\) such that for all \(i\) and \(t\) the maps \(f^i_t = f^i(-,t)\colon \R^n\to\R\) are affine linear forms and satisfy
            \begin{equation*}
                \rho_t(Q_0) = \left\{ \, x\in \R^n \,\middle|\,
                    f^i_t(x) \ge 0 \text{ for all \(i\)}
                \,\right\}.
            \end{equation*}
    \end{enumerate}
    Hence, the images \(\rho_t(Q_0)\) are all polyhedra and we write \(Q_t\) for \(\rho_t(Q_0)\) and say \((Q_t)_{t\in [0,1]}\) is a continuous deformation when the accompanying maps \(\rho\) and \(f^i\) are clear from the context.
\end{definition}

Note that a continuous deformation of polyhedra as defined here consists of both a map moving the points around \emph{and} a continuous description in terms of inequalities for all \(t\in[0,1]\).

\begin{definition} \label{def:cont-deg}
    A continuous deformation \((Q_t)_{t\in [0,1]}\) as in \Cref{def:cont-deform} is called a \emph{continuous degeneration} if for all \(x\in Q_0\), \(t<1\) and \(i=1,\dots,k\) we have \(f^i_t(\rho_t(x)) = 0\) if and only if \(f^i_0(x) = 0\).
\end{definition}

From this definition we immediately obtain the following.

\begin{proposition}
    If \((Q_t)_{t\in [0,1]}\) is a continuous degeneration, the polyhedra \(Q_t\) for \(t<1\) are all combinatorially equivalent and \(\rho_t\) preserves faces and their incidence structure.
\end{proposition}

\begin{proof}
    Let the data of the continuous degeneration be given as in \Cref{def:cont-deform}.
    For \(y\in Q_t\) denote by \(\FI_t(y)\) the set of all \(i\in[k]\) such that \(f^i_t(y) = 0\).
    The set of all \(\FI_t(y)\) for \(y\in Q_t\) ordered by reverse inclusion is isomorphic to \(\FF(Q_t)\setminus\{\varnothing\}\) since relative interiors of faces of \(Q_t\) correspond to regions of constant \(\FI_t\).

    Since for all \(x\in Q_0\), \(t<1\) and \(i=1,\dots,k\) we have \(f^i_t(\rho_t(x))=0\) if and only if \(f^i_0(x)=0\), the sets \(\FI_t(\rho_t(x))\) are fixed for \(t<1\) and hence \(\rho_t\) preserves the face structure.
\end{proof}

We continue by illustrating the definition of continuous degenerations in an example before proceeding with the general theory.

\begin{example} \label{ex:contdeg}
    For \(t\in[0,1]\) let \(Q_t\subseteq\R^2\) be the polytope defined by the inequalities \(0\le x_1 \le 2\), \(0\le x_2\) as well as
    \begin{align*}
        x_2 &\le (1-t)x_1 + 1,\quad\text{and} \\
        x_2 &\le (1-t)(2-x_1) + 1.
    \end{align*}
    \begin{figure}
        \centering
        \subcaptionbox[]{\(Q_0\)\label{subfig:contdeg-a}}[0.32\textwidth][c]{\centering
        \begin{tikzpicture}[scale=1.2,baseline={([yshift=-.5ex]current bounding box.center)}]
            \draw (-.5,0) -- (2.4,0) node [right] {\(x_1\)};
            \draw (0,-.5) -- (0,2.4) node [above] {\(x_2\)};
            \fill[black!6,draw=black,thick] (0,0) -- (2,0) -- (2,1) -- (1,2) -- (0,1) -- cycle;
            \draw[thick,blue,dashed] (1,0) -- (1,2.4);
            \draw (1,0) node[below] {\(1\)};
            \draw (2,0) node[below] {\(2\)};
            \draw (0,1) node[left] {\(1\)};
            \draw (0,2) node[left] {\(2\)};
        \end{tikzpicture}}
        \hfill
        \subcaptionbox[]{\(Q_{\frac 1 2}\)\label{subfig:contdeg-b}}[0.32\textwidth][c]{\centering
        \begin{tikzpicture}[scale=1.2,baseline={([yshift=-.5ex]current bounding box.center)}]
            \draw (-.5,0) -- (2.4,0) node [right] {\(x_1\)};
            \draw (0,-.5) -- (0,2.4) node [above] {\(x_2\)};
            \fill[black!6,draw=black,thick] (0,0) -- (2,0) -- (2,1) -- (1,1.5) -- (0,1) -- cycle;
            \draw[thick,blue,dashed] (1,0) -- (1,2.4);
            \draw (1,0) node[below] {\(1\)};
            \draw (2,0) node[below] {\(2\)};
            \draw (0,1) node[left] {\(1\)};
            \draw (0,2) node[left] {\(2\)};
        \end{tikzpicture}}
        \hfill
        \subcaptionbox[]{\(Q_1\)\label{subfig:contdeg-c}}[0.32\textwidth][c]{\centering
        \begin{tikzpicture}[scale=1.2,baseline={([yshift=-.5ex]current bounding box.center)}]
            \draw (-.5,0) -- (2.4,0) node [right] {\(x_1\)};
            \draw (0,-.5) -- (0,2.4) node [above] {\(x_2\)};
            \fill[black!6,draw=black,thick] (0,0) -- (2,0) -- (2,1) -- (1,1) -- (0,1) -- cycle;
            \draw[thick,blue,dashed] (1,0) -- (1,2.4);
            \draw (1,0) node[below] {\(1\)};
            \draw (2,0) node[below] {\(2\)};
            \draw (0,1) node[left] {\(1\)};
            \draw (0,2) node[left] {\(2\)};
        \end{tikzpicture}}
        \caption[Continuous degeneration from \Cref{ex:contdeg}]{The polytopes in the continuous degeneration from \Cref{ex:contdeg} for \(t=0\), \(t=\tfrac 1 2\) and \(t=1\).}
        \label{fig:contdeg}
    \end{figure}
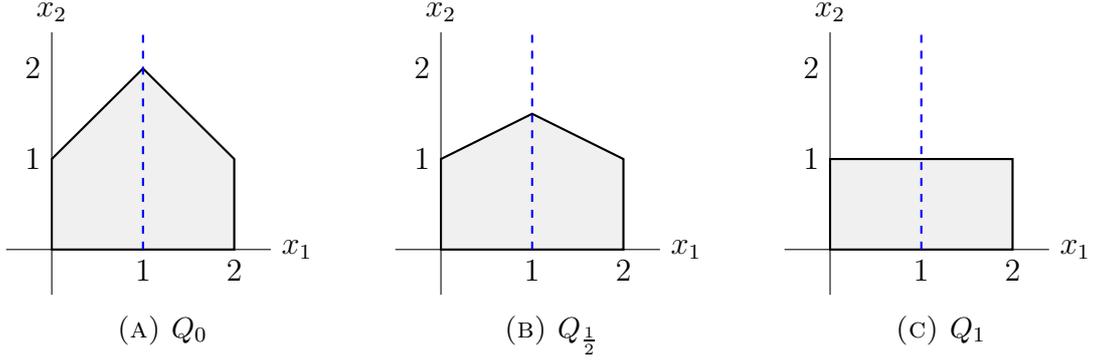

    For \(t=0\), \(t=\tfrac 1 2\) and \(t=1\) we have illustrated the polytope in \Cref{fig:contdeg}.
    Together with the map \(\rho_t\colon Q_0\to \R^2\) given by \(\rho_t(x)_1=x_1\) for all \(t\) and
    \begin{equation*}
        \rho_t(x)_2 =
        \begin{cases}
            x_2 \, \frac{(1-t)x_1 + 1}{x_1+1} & \text{for \(x_1 \le 1\),} \\
            x_2 \, \frac{(1-t)(2-x_1) + 1}{(2-x_1)+1} & \text{for \(x_1 \ge 1\)} \\
        \end{cases}
    \end{equation*}
    we obtain a continuous degeneration.
    Starting from the pentagon in \Cref{subfig:contdeg-a} at \(t=0\) we see increasingly compressed pentagons with the two top edges becoming more flat-angled until ending up with the rectangle in \Cref{subfig:contdeg-c} at \(t=1\).
    The map \(\rho_t\) just scales the \(x_2\) coordinates accordingly, preserving the face structure for \(t<1\).
\end{example}

The key result on continuous degenerations that will allow conclusions on the face structure of degenerations is that during a continuous degeneration, relative interiors of faces always map into relative interiors of faces.
In other words, continuous degenerations can not ``fold'' faces of \(Q_0\) so they split into different faces of \(Q_1\), but only ``straighten'' some adjacent faces of \(Q_0\) to become one face of \(Q_1\).

\begin{proposition} \label{prop:degeneration}
    Let \((Q_t)_{t\in[0,1]}\) be a continuous degeneration of polyhedra.
    Whenever \(F\) is a face of \(Q_0\), there is a unique face \(G\) of \(Q_1\) such that
    \begin{equation*}
        \rho_1(\relint F) \subseteq \relint G.
    \end{equation*}
\end{proposition}

\begin{proof}
    As in the previous proof, let \(\FI_t(y)\) denote the set of indices \(i\in[k]\) such that \(f^i_t(y)=0\).
    Using these incidence sets we may rephrase the proposition as follows: whenever \(x,x'\in Q_0\) satisfy \(\FI_0(x)=\FI_0(x')\), they also satisfy \(\FI_1(\rho_1(x)) = \FI_1(\rho_1(x'))\).
    
    Let \(F\) be the face of \(Q_0\) having both \(x\) and \(x'\) in its relative interior and assume there exists a \(j\in \FI_1(\rho_1(x)) \setminus \FI_1(\rho_1(x'))\) for sake of contradiction.
    Hence, we have \(f^j_1(\rho_1(x))=0\) while \(f^j_1(\rho_1(x'))>0\).
    Let \(d\) denote the dimension of \(F\) then \(\relint F\) is a (topological) manifold of dimension \(d\).
    Since \(\rho_1\) is an embedding, \(\rho_1(\relint F)\) is a manifold of dimension \(d\) as well.
    Since the affine hull of \(\rho_t(\relint F)\) is of dimension \(d\) for all \(t<1\), we conclude that the affine hull of \(\rho_1(\relint F)\) is of dimension at most \(d\).
    To see this, take any \(d+1\) points \(y_0,\dots,y_d\) in \(\rho_1(\relint F)\). Their images \(\rho_t(\rho_1^{-1}(y_0)),\dots,\rho_t(\rho_1^{-1}(y_d))\) in \(\rho_t(\relint F)\) are affinely dependent for \(t<1\), so they have to be affinely dependent for \(t=1\) as well by the continuity of \(\rho\) in \(t\).

    But as \(\rho_1(\relint F)\) is a manifold of dimension \(d\), we conclude that its affine hull has dimension exactly \(d\) and \(\rho_1(\relint F)\) is an open subset of its affine hull.
    Given that both \(\rho_1(x)\) and \(\rho_2(x')\) are points in \(\rho_1(\relint F)\), we conclude that there exists an \(\varepsilon>0\) such that the point
    \begin{equation*}
        z = \rho_1(x) + \varepsilon(\rho_1(x) - \rho_1(x'))
    \end{equation*}
    is still contained in \(\rho_1(\relint F)\).
    In particular, \(z\in Q_1\).
    However, since \(f^j_1\) is an affine linear form, we have
    \begin{equation*}
        f^j_1(z) = (1+\varepsilon) f^j_1(\rho_1(x)) - \varepsilon f^j_1(\rho_1(x')) < 0.
    \end{equation*}
    This contradicts \(z\in Q_1\), which finishes the proof.
\end{proof}

The consequence of \Cref{prop:degeneration} is that continuous degenerations induce maps between face lattices.

\begin{corollary} \label{cor:degeneration}
    When \((Q_t)_{t\in[0,1]}\) is a continuous degeneration of polyhedra, we have a surjective order-preserving map of face lattices
    \begin{equation*}
        \dg\colon \FF(Q_0) \longrightarrow \FF(Q_1)
    \end{equation*}
    determined by the property
    \begin{equation*}
        \rho_1(\relint F) \subseteq \relint \dg(F).
    \end{equation*}
    for non-empty \(F\) and \(\dg(\varnothing)=\varnothing\).
    Furthermore, the map satisfies \(\dim(\dg(F))\ge \dim F\) for all \(F\in\FF(Q_0)\). \qed
\end{corollary}

We will refer to the map in \Cref{cor:degeneration} as the \emph{degeneration map}.
Before coming back to marked poset polyhedra, we finish with a result on the $f$-vectors of continuous degenerations.

\begin{proposition} \label{prop:degenerate-f-vector}
    Let \((Q_t)_{t\in[0,1]}\) be a continuous degeneration of polyhedra.
    We have \(f_i(Q_1) \le f_i(Q_0)\) for all \(i\).
\end{proposition}

\begin{proof}
    Let \(G\) be an \(i\)-dimensional face of \(Q_1\).
    We claim that there is at least one \(i\)-dimensional face \(F\) of \(Q_0\) such that \(\dg(F)=G\).
    Since every polyhedron is the disjoint union of the relative interiors of its faces and \(\rho_1\) is a bijection, we have
    \begin{equation*}
        \relint G = \bigsqcup_{F\in\dg^{-1}(G)} \rho_1(\relint F).
    \end{equation*}
    Since \(\relint G\) is a manifold of dimension \(\dim G\) and each \(\rho_1(\relint F)\) is a manifold of dimension \(\dim F \le \dim G\), there has to be at least one \(F\in\dg^{-1}(G)\) of the same dimension as \(G\).
\end{proof}

\subsection{Continuous Degenerations in the Universal Family}

We are now ready to apply the concept of continuous degenerations to the universal family of marked poset polyhedra.
Let us first identify for which pairs of parameters \(u,u'\in\smash{[0,1]^{\tilde P}}\) we expect to have a continuous degeneration from \(\OO_u(P,\lambda)\) to \(\OO_{u'}(P,\lambda)\) and then specify the deformation precisely.

\begin{definition}
    Let \(u\in\smash{[0,1]^{\tilde P}}\) and let \(I\subseteq\tilde P\) be the set of indices \(p\), such that \(u_p\in\{0,1\}\).
    Any \(u'\in\smash{[0,1]^{\tilde P}}\) such that \(u'_p=u_p\) for \(p\in I\)  is called a \emph{degeneration} of \(u\).
\end{definition}

\begin{proposition} \label{prop:mpp-degeneration}
    Let \(u'\) be a degeneration of \(u\). The map
    \begin{align*}
        \rho \colon \OO_u(P,\lambda) \times [0,1] &\longrightarrow \R^P, \\
        (x, \xi) &\longmapsto \theta_{u,\xi u'+(1-\xi)u}(x)
    \end{align*}
    is a continuous degeneration with the accompanying affine linear forms given by the equations and inequalities in \Cref{def:mpp} for \(t=\xi u'+(1-\xi)u\).
\end{proposition}

\begin{proof}
    The map \(\rho\) together with the affine linear forms given by \Cref{def:mpp} is a continuous deformation by \Cref{thm:univ-transfer}.
    The fact that \(\rho\) is a continuous degeneration follows from \Cref{prop:ineq-pullback}.
\end{proof}

Now the machinery of continuous degenerations immediately yields degeneration maps and results on the $f$-vectors of marked poset polyhedra.

\begin{corollary} \label{cor:f-vector-degeneration}
    Let \(u,u'\in\smash{[0,1]^{\tilde P}}\) such that \(u'\) is a degeneration of \(u\). The continuous degeneration in \Cref{prop:mpp-degeneration} yields a degeneration map \(\dg_{u,u'}\colon \OO_u(P,\lambda) \to \OO_{u'}(P,\lambda)\) in the sense of \Cref{cor:degeneration}.
    In particular, the $f$-vectors satisfy
    \begin{equation*}
        f_i(\OO_{u'}(P,\lambda)) \le f_i(\OO_u(P,\lambda)) \quad\text{for all \(i\).}
    \end{equation*}

    Furthermore, given a degeneration \(u''\) of \(u'\), the degeneration maps satisfy
    \begin{equation*}
        \dg_{u,u''} = \dg_{u',u''} \circ \dg_{u,u'}.
    \end{equation*}
\end{corollary}

\begin{proof}
    After applying \Cref{prop:degeneration}, \Cref{cor:degeneration} and \Cref{prop:degenerate-f-vector} to the situation at hand, all that remains to be proven is the statement about compositions of degeneration maps.
    This is an immediate consequence of \(\theta_{u,u''} = \theta_{u',u''} \circ \theta_{u,u'}\).
\end{proof}

\section{Tropical Arrangements and Subdivisions}
\label{sec:tropical}

As discussed in \Cref{sec:mop-JS-subdivision}, the marked order polyhedron \(\OO(P,\lambda)\) comes with a subdivision \(\SSc\) into products of simplices and simplicial cones.
Since the parametrized transfer map \(\varphi_t\) is linear on each cell of \(\SSc\), we have a transferred subdivision \(\SSc_t\) of \(\OO_t(P,\lambda)\) for all \(t\in\smash{[0,1]^{\tilde P}}\).

In this section we introduce a coarsening of \(\SSc\) into linearity regions of \(\varphi\), obtained by intersecting \(\OO(P,\lambda)\) with the cells in a tropical hyperplane arrangement determined by \((P,\lambda)\).
Our main reason to consider this subdivision is a result in \Cref{sec:generic-vertices}, where we will show that the vertices of generic marked poset polyhedra are given by the vertices in this subdivision and hence can be obtained by first subdividing the marked order polyhedron to then transfer the vertices in the subdivision.
The notation we use here is close to \cite{FR15}, where the combinatorics of tropical hyperplane arrangements are discussed in detail.

\subsection{Tropical Hyperplane Arrangements}

In tropical geometry, the usual ring structure \((\R,+,\cdot)\) we use for Euclidean geometry is replaced by the \emph{tropical semiring} \((\R\cup\{-\infty\},\oplus,\odot)\), where \(a\oplus b=\max(a,b)\), \(a\odot b=a+b\) and \(-\infty\) is the identity with respect to \(\oplus\).
Hence, a \emph{tropical polynomial} is a convex piecewise-linear function in ordinary terms:
\begin{equation*}
    \bigoplus_{a\in\N^n} c_a \odot x_1^{\odot a_1} \odot \cdots \odot x_n^{\odot a_n} = \max \big\{\, a_1 x_1 + \cdots + a_n x_n + c_a : a\in\N^n \,\big\}.
\end{equation*}
Given a \emph{tropical linear form}
\begin{equation*}
    \alpha = \bigoplus_{i=1}^n c_i \odot x_i = \max \big\{\, x_i + c_i : i=1,\dots,n \,\big\},
\end{equation*}
where some---but not all---coefficients are allowed to be \(-\infty\), one defines a \emph{tropical hyperplane} \(H_\alpha\) consisting of all \(x\in\R^n\) such that \(\alpha\) is non-differentiable at \(x\) or equivalently, the maximum in \(\alpha(x)\) is attained at least twice.

We may pick some of the coefficients \(c_i\) to be \(-\infty\) to obtain tropical linear forms only involving some of the coordinates.
For example when \(n=3\) we could have
\begin{equation*}
    \alpha = (1 \odot x_1) \oplus (3 \odot x_2) = \max\{1+x_1,3+x_2\} = (1 \odot x_1) \oplus (3 \odot x_2) \oplus (-\infty \odot x_3)
\end{equation*}
and the tropical hyperplane \(H_\alpha\) would just be the usual hyperplane \(1 + x_1 = 3 + x_2\).
Given a tropical hyperplane, one obtains a polyhedral subdivision of \(\R^n\) with facets the linearity regions of \(\alpha\) and the skeleton of codimension \(1\) being \(H_\alpha\) as follows:
for a tropical hyperplane \(H=H_\alpha\) in \(\R^n\) define the \emph{support} \(\supp(H)\) as the set of all \(i\in[n]\) such that the coefficient \(c_i\) is different from \(-\infty\) in \(\alpha\).
For any non-empty subset \(L\subseteq \supp(H)\) we have a cell
\begin{equation*}
    F_L(H) = \left\{ \, x\in\R^n : \text{\(c_l + x_l = \max_{i\in\supp(H)} (x_i + c_i)\) for all \(l\in L\)} \, \right\}.
\end{equation*}
The facets \(F_{\{l\}}\) for \(l\in\supp(H)\) are the linearity regions of \(\alpha\) and the cells \(F_L\) for \(|L|\ge 2\) form a subdivision of \(H_\alpha\).
Given any \(x\in\R^n\), we define its \emph{signature} \(\sig_H(x)\) as the unique \(L\subseteq\supp(H)\) such that \(x\) is in the relative interior of \(F_L\).
Equivalently, the signature of \(x\) is the set of indices achieving the maximum in \(\alpha(x)\),
\begin{equation*}
    \sig_H(x) = \argmax_{i\in\supp(H)}\ (x_i + c_i).
\end{equation*}
Using this terminology, we may also describe \(F_L(H)\) as the set of all points \(x\in\R^n\) with \(L\subseteq\sig_H(x)\)

Now let \(\HH=\{H_1,H_2,\dots,H_m\}\) be a \emph{tropical hyperplane arrangement}, that is, each \(H_i\) is a tropical hyperplane \(H_{\alpha_i}\subseteq\R^n\) for a tropical linear form \(\alpha_i\).
The common refinement \(\TT(\HH)\) of the polyhedral subdivision of \(H_1, H_2,\dots, H_m\) gives a polyhedral subdivision of \(\R^n\) whose facets are the largest regions on which all \(\alpha_i\) are linear and whose \((n-1)\)-skeleton is a subdivision of \(\bigcup H_i\).
To each \(x\in\R^n\) we associate the \emph{tropical covector} \(\tc(x)\colon [m]\to 2^{[n]}\) recording the signatures with respect to all hyperplanes, that is
\begin{equation*}
    \tc(x) = \left( \sig_{H_1}(x), \sig_{H_2}(x), \dots, \sig_{H_m}(x) \right).
\end{equation*}
Hence, the cells of \(\TT(\HH)\) are enumerated by the appearing tropical covectors when \(x\) varies over all points in \(\R^n\).
The set of all these tropical covectors is called the \emph{combinatorial type} of \(\HH\) and denoted \(\TC(\HH)\).
For each \(\tau\in\TC(\HH)\) the corresponding cell is given by
\begin{equation*}
    F_\tau = \bigcap_{i=1}^m F_{\tau_i}(H_i)
\end{equation*}
and its relative interior consists of all \(x\in\R^n\) such that \(\tc(x)=\tau\).

To digest all these definitions, let us look at a small example before using the introduced terminology to define a subdivision of marked poset polyhedra.

\begin{example} \label{ex:tropical-arrangement}
    Let \(n=3\) and consider the following tropical linear forms:
    \begin{alignat*}{2}
        \alpha_1 &= \left( (-2)\odot x_1 \right) \oplus \left( (-1)\odot x_2 \right) \oplus \left( 0 \odot x_3 \right) &&= \max\{x_1-2,x_2-1,x_3\}, \\
        \alpha_2 &= \left( (-2)\odot x_1 \right) \oplus \left( 0\odot x_2 \right) \oplus \left( (-\infty) \odot x_3 \right) &&= \max\{x_1-2,x_2\}, \\
        \alpha_3 &= \left( (-1)\odot x_1 \right) \oplus \left( (-\infty)\odot x_2 \right) \oplus \left( 0 \odot x_3 \right) &&= \max\{x_1-1,x_3\}.
    \end{alignat*}
    Let \(\HH=\{H_1,H_2,H_3\}\) be the tropical hyperplane arrangement with \(H_i\) given by \(\alpha_i\) for \(i=1,2,3\).
    The supports of the three hyperplanes are \(\supp(H_1)=\{1,2,3\}\), \(\supp(H_2)=\{1,2\}\) and \(\supp(H_3)=\{1,3\}\).
    Since tropical hyperplanes are invariant under translations along the all-one vector \((1,1,\dots,1)\in\R^n\), we obtain a faithful picture of the subdivision \(\TT(\HH)\) by just looking at the slice \(x_n=0\).
    This is done in \Cref{fig:tropical-arrangement} for the example at hand with some of the appearing tropical covectors listed.
    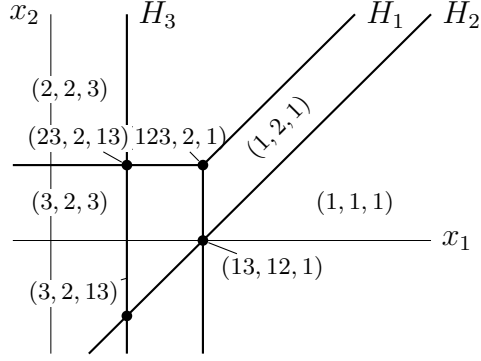
\begin{figure}
        \centering
        \begin{tikzpicture}
            \draw[very thin] (-.5,0) -- (5,0) node[right] {\(x_1\)};
            \draw[very thin] (0,-1.5) -- (0,3) node[left] {\(x_2\)};
            \draw[thick]
                (1,-1.5) -- (1,3) node[right] {\(H_3\)}
                (.5,-1.5) -- (5,3) node[right] {\(H_2\)}
                (2,-1.5) -- (2,1) -- (-.5,1)
                (2,1) -- (4,3) node[right] {\(H_1\)};
            \tikzset{vertex/.style={fill,circle,minimum size=4pt,inner sep=0},tc/.style={font=\scriptsize,inner sep=.5pt}}
            \path (2,0) node[vertex] (a) {} node[tc,below right=5pt,xshift=1pt] (atc) {\((13,12,1)\)};
            \path (2,1) node[vertex] (b) {} node[tc,above=5pt,xshift=-10pt] (btc) {\((123,2,1)\)};
            \path (1,1) node[vertex] (c) {} node[tc,above=5pt,xshift=-18pt,fill=white] (ctc) {\((23,2,13)\)};
            \path (1,-1) node[vertex] {};
            \draw (atc) -- (a);
            \draw (btc) -- (b);
            \draw (ctc) -- (c);
            \node[tc,rotate=45] at (3,1.5) {\((1,2,1)\)};
            \node[tc] at (4,.5) {\((1,1,1)\)};
            \node[tc,fill=white] at (.25,2) {\((2,2,3)\)};
            \node[tc,fill=white] at (.25,.5) {\((3,2,3)\)};
            \node[tc,fill=white] (etc) at (.3,-.7) {\((3,2,13)\)};
            \draw (etc) -- (1,-.5);
        \end{tikzpicture}
        \caption[Tropical hyperplane arrangement and covectors from \Cref{ex:tropical-arrangement}]{The tropical hyperplane arrangement from \Cref{ex:tropical-arrangement} sliced at \(x_3=0\) with some of the appearing tropical covectors listed.}
        \label{fig:tropical-arrangement}
    \end{figure}
\end{example}

\subsection{The Tropical Subdivision}

We are now ready to introduce the \emph{tropical subdivision} of marked poset polyhedra.
As before, let \((P,\lambda)\) be a marked poset with at least all minimal elements marked.
The transfer maps \(\varphi_t\) of \Cref{thm:univ-transfer} give rise to the tropical linear forms
\begin{equation*}
    \alpha_p = \max_{q\prec p} x_q = \bigoplus_{q\prec p} x_q \quad\text{for \(p\in\tilde P\)}.
\end{equation*}
When \(p\) is not covering at least two elements, the tropical linear form \(\alpha_p\) has just one term and defines an empty tropical hyperplane since the maximum can never be achieved twice.
Hence, let \(R\) denote the set of all \(p\in\tilde P\) covering at least two elements and define a tropical hyperplane arrangement \(\HH(P,\lambda)\) in \(\R^P\) with tropical hyperplanes \(H_p = H_{\alpha_p}\) for all \(p\in R\).
By construction, the facets of \(\TT(\HH(P,\lambda))\) are the linearity regions of \(\varphi_t\) for \(t\in\smash{(0,1]^{\tilde P}}\).

The reason this subdivision will help study the combinatorics of marked poset polytopes is the following: by \Cref{prop:ineq-pullback} the combinatorics of \(\OO_t(P,\lambda)\) can be determined by pulling points back to \(\OO(P,\lambda)\) and looking at the relation \(\dashv_x\).
But for \(r\in R\) and \(p\in P\) we have \(p\dashv_x r\) if and only if \(p\in\tc(x)_r\), so the information encoded in \(\dashv_x\) is equivalent to knowing the minimal cell of \(\TT(\HH(P,\lambda))\) containing \(x\).

Using this tropical hyperplane arrangement, we can define a polyhedral subdivision of \(\OO(P,\lambda)\).

\begin{definition}
    Let \(\TT(\HH(P,\lambda))\) be the polyhedral subdivision of \(\R^P\) associated to the marked poset \((P,\lambda)\).
    The \emph{tropical subdivision} \(\TT(P,\lambda)\) of \(\OO(P,\lambda)\) is given by the intersection of faces of \(\OO(P,\lambda)\) with the faces of \(\TT(\HH(P,\lambda))\):
    \begin{equation*}
        \TT(P,\lambda) = \left\{ \, F\cap G \,\middle|\, F\in\FF(\OO(P,\lambda)), G\in\TT(\HH(P,\lambda)) \,\right\}.
    \end{equation*}
    For \(t\in\smash{[0,1]^{\tilde P}}\) define the tropical subdivision of \(\OO_t(P,\lambda)\) as
    \begin{equation*}
        \TT_t(P,\lambda) = \left\{ \varphi_t(Q) \,\middle|\, Q\in \TT(P,\lambda) \right\}.
    \end{equation*}
\end{definition}

Note that \(\TT_t(P,\lambda)\) is polyhedral subdivision of \(\OO_t(P,\lambda)\) since \(\varphi_t\) is linear on each \(G\in\TT(P,\lambda)\) by construction.
In particular, \(\TT_t(P,\lambda)\) is a coarsening of the subdivision \(\SSc_t\) into products of simplices and simplicial cones.

\section{Vertices in the Generic Case}
\label{sec:generic-vertices}

Using the tropical subdivision from \Cref{sec:tropical} and the concept of continuous degenerations from \Cref{sec:cont-deg}, we are ready to prove a theorem describing the vertices of generic marked poset polyhedra.

\begin{theorem} \label{thm:generic-vertices}
    The vertices of a generic marked poset polyhedron \(\OO_t(P,\lambda)\) with \(t\in(0,1)^{\tilde{P}}\) are exactly the vertices in its tropical subdivision \(\TT_t(P,\lambda)\).
\end{theorem}

As a consequence, the vertices of the generic marked poset polyhedron can be obtained by subdividing the marked order polyhedron using the associated tropical subdivision and transferring the obtained vertices via the transfer map \(\varphi_t\) to \(\OO_t(P,\lambda)\).
Furthermore, even for arbitrary \(t\in[0,1]^{\tilde{P}}\), the set of points obtained this way will always contain the vertices of \(\OO_t(P,\lambda)\).

Before proceeding with the proof of \Cref{thm:generic-vertices} let us illustrate the situation with an example.
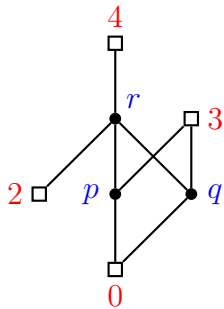
\begin{figure}
    \centering
    \begin{tikzpicture}[scale=1]
        \path (0,3) node[posetelmm] (4) {} node[above=2pt,marking] {\(4\)};
        \path (0,1) node[posetelm] (p) {} node[left=2pt,elmname] {\(p\)};
        \path (1,1) node[posetelm] (q) {} node[right=2pt,elmname] {\(q\)};
        \path (0,2) node[posetelm] (r) {} node[above right,elmname] {\(r\)};
        \path (0,0) node[posetelmm] (0) {} node[below=2pt,marking] {\(0\)};
        \path (1,2) node[posetelmm] (3) {} node[right=2pt,marking] {\(3\)};
        \path (-1,1) node[posetelmm] (2) {} node[left=2pt,marking] {\(2\)};
        \draw[covrel]
              (0) -- (p) -- (r) -- (4)
              (0) -- (q) -- (r) -- (2)
              (p) -- (3) -- (q);
    \end{tikzpicture}
    \caption[Marked poset from \Cref{ex:vertices}]{The marked poset from \Cref{ex:vertices}}
    \label{fig:vertex-poset}
\end{figure}

\begin{example} \label{ex:vertices}
    Let \((P,\lambda)\) be the marked poset given in \Cref{fig:vertex-poset}.
    The hyperplane arrangement \(\HH(P,\lambda)\) consists of just one tropical hyperplane given by the tropical linear form
    \begin{equation*}
        \alpha_r = \max \{x_2, x_p, x_q\} = x_2 \oplus x_p \oplus x_q.
    \end{equation*}
    It divides the space \(\R^P\) into three regions where either \(x_2\), \(x_p\) or \(x_q\) is maximal among the three coordinates.
    Intersecting this subdivision with \(\OO(P,\lambda)\) we obtain the the tropical subdivision shown in \Cref{subfig:vertex-O}, where the hyperplane itself is shaded in red.
    We see the 11 vertices of the polytope depicted in green and 3 additional vertices of the tropical subdivision that are not vertices of \(\OO(P,\lambda)\) in red.
    Since \(t_p\) and \(t_q\) are irrelevant for the affine type of \(\OO_t(P,\lambda)\) by \Cref{prop:irrelevant-p}---and in fact only get multiplied by \(0\) in the projected transfer map \(\varphi_t\)---we only need to consider the parameter \(t_r\).
    In \Cref{subfig:vertex-gen} we see the tropical subdivision of \(\OO_t(P,\lambda)\) for \(t_r=\tfrac 1 2\).
    Now all vertices that appear in the subdivision are green, i.e., they are vertices of the polytope, as stated in \Cref{thm:generic-vertices}.
    When \(t_r=1\), we obtain the tropic subdivision of the marked chain polytope \(\CC(P,\lambda)\) as shown in \Cref{subfig:vertex-C}.
    Again, some of the vertices in the subdivision are not vertices of the polytope.
    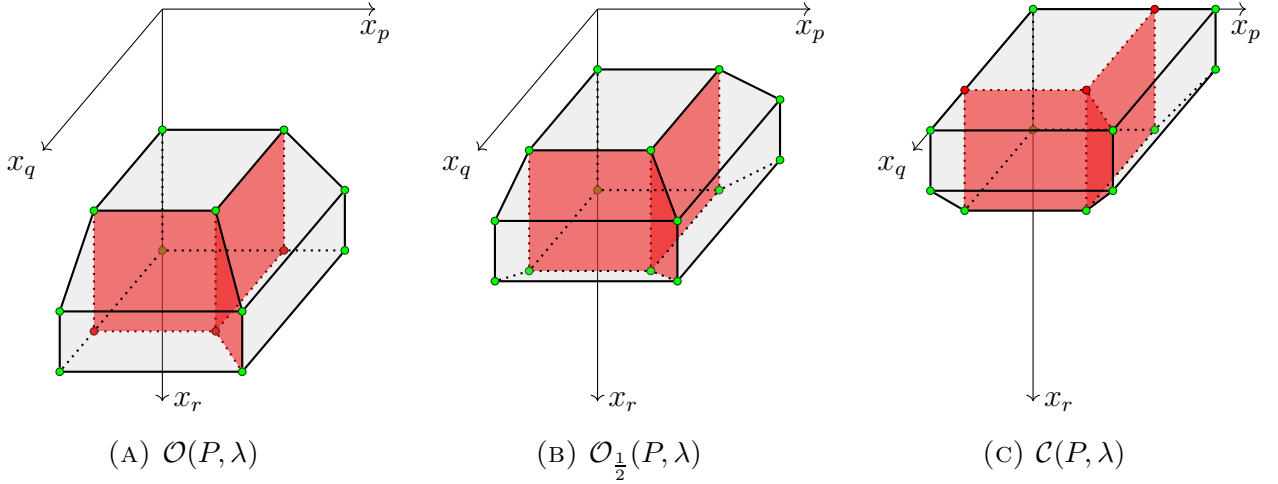
\begin{figure}
        \centering
        \makebox[\textwidth][c]{
        \subcaptionbox[]{\(\OO(P,\lambda)\)\label{subfig:vertex-O}}[0.33\textwidth][c]{
        \begin{tikzpicture}[scale=1,baseline={(current bounding box.center)},
                x={(0:.8cm)},
                y={(230:.7cm)},
                z={(270:.8cm)},
                back/.style={dotted},
                edge/.style={thick},
                edgeb/.style={thick,color=red!50!black},
                facet/.style={fill=black!30,fill opacity=0.2},
                facetb/.style={fill=red,fill opacity=0.6},
                vertex/.style={inner sep=1pt,circle,draw=green!25!black,fill=green,anchor=base},
                vertexb/.style={inner sep=1pt,circle,draw=red!25!black,fill=red,anchor=base},
                axis/.style={->}
                ]
            \draw[axis] (0,0,0) -- (3.5,0,0) node[below] {\(x_p\)};
            \draw[axis] (0,0,0) -- (0,3.5,0) node[below left=-2pt] {\(x_q\)};
            \draw[axis] (0,0,0) -- (0,0,6.5) node[right] {\(x_r\)};
            \coordinate (A) at (0,0,2);
            \coordinate (B) at (2,0,2);
            \coordinate (C) at (2,2,2);
            \coordinate (D) at (0,2,2);
            \coordinate (E) at (3,0,3);
            \coordinate (F) at (3,3,3);
            \coordinate (G) at (0,3,3);
            \coordinate (H) at (0,0,4);
            \coordinate (I) at (3,0,4);
            \coordinate (J) at (3,3,4);
            \coordinate (K) at (0,3,4);
            \coordinate (L) at (2,0,4);
            \coordinate (M) at (2,2,4);
            \coordinate (N) at (0,2,4);

            \node[vertex] (Hn) at (H) {}; 

            \fill[facetb] (B) -- (C) -- (M) -- (L) -- cycle;
            \fill[facetb] (C) -- (F) -- (J) -- (M) -- cycle;
            \fill[facetb] (D) -- (C) -- (M) -- (N) -- cycle;

            \node[vertexb] (Ln) at (L) {};
            \node[vertexb] (Mn) at (M) {};
            \node[vertexb] (Nn) at (N) {};

            \fill[facet] (A) -- (B) -- (C) -- (D) -- cycle;
            \fill[facet] (B) -- (C) -- (F) -- (E) -- cycle;
            \fill[facet] (C) -- (D) -- (G) -- (F) -- cycle;
            \fill[facet] (E) -- (F) -- (J) -- (I) -- cycle;
            \fill[facet] (F) -- (G) -- (K) -- (J) -- cycle;

            \node[vertex] (Dn) at (D) {}; 
            \node[vertex] (Gn) at (G) {}; 
            \node[vertex] (An) at (A) {}; 
            \node[vertex] (Bn) at (B) {}; 
            \node[vertex] (Cn) at (C) {}; 
            \node[vertex] (En) at (E) {}; 
            \node[vertex] (Fn) at (F) {}; 
            \node[vertex] (In) at (I) {}; 
            \node[vertex] (Jn) at (J) {}; 
            \node[vertex] (Kn) at (K) {};

            \draw[edgeb,back]
                (Mn) -- (Cn)
                (Nn) -- (Dn)
                (Ln) -- (Mn)
                (Mn) -- (Nn)
                (Mn) -- (Jn)
                (Ln) -- (Bn)
                ;

            \draw[edge,back]
                (An) -- (Hn)
                (Hn) -- (Ln)
                (Nn) -- (Hn)
                (Ln) -- (In)
                (Kn) -- (Nn)
                ;
            \draw[edge]
                (Cn) -- (Dn)
                (Dn) -- (An)
                (Dn) -- (Gn)
                (Fn) -- (Gn)
                (Gn) -- (Kn)
                (An) -- (Bn)
                (Bn) -- (Cn)
                (Bn) -- (En)
                (Cn) -- (Fn)
                (En) -- (Fn)
                (En) -- (In)
                (Fn) -- (Jn)
                (In) -- (Jn)
                (Jn) -- (Kn)
                ;
        \end{tikzpicture}}
        \hskip 1.5em
        \subcaptionbox[]{\(\OO_{\frac 1 2}(P,\lambda)\)\label{subfig:vertex-gen}}[0.33\textwidth][c]{
        \begin{tikzpicture}[scale=1,baseline={(current bounding box.center)},
                x={(0:.8cm)},
                y={(230:.7cm)},
                z={(270:.8cm)},
                back/.style={dotted},
                edge/.style={thick},
                edgeb/.style={thick,color=red!50!black},
                facet/.style={fill=black!30,fill opacity=0.2},
                facetb/.style={fill=red,fill opacity=0.6},
                vertex/.style={inner sep=1pt,circle,draw=green!25!black,fill=green,anchor=base},
                vertexb/.style={inner sep=1pt,circle,draw=red!25!black,fill=red,anchor=base},
                axis/.style={->}
                ]
            \draw[axis] (0,0,0) -- (3.5,0,0) node[below] {\(x_p\)};
            \draw[axis] (0,0,0) -- (0,3.5,0) node[below left=-2pt] {\(x_q\)};
            \draw[axis] (0,0,0) -- (0,0,6.5) node[right] {\(x_r\)};
            \coordinate (A) at (0,0,1);
            \coordinate (B) at (2,0,1);
            \coordinate (C) at (2,2,1);
            \coordinate (D) at (0,2,1);
            \coordinate (E) at (3,0,3/2);
            \coordinate (F) at (3,3,3/2);
            \coordinate (G) at (0,3,3/2);
            \coordinate (I) at (3,0,5/2);
            \coordinate (J) at (3,3,5/2);
            \coordinate (K) at (0,3,5/2);
            \coordinate (L) at (2,0,3);
            \coordinate (H) at (0,0,3);
            \coordinate (M) at (2,2,3);
            \coordinate (N) at (0,2,3);

            \node[vertex] (Hn) at (H) {}; 

            \fill[facetb] (B) -- (C) -- (M) -- (L) -- cycle;
            \fill[facetb] (C) -- (F) -- (J) -- (M) -- cycle;
            \fill[facetb] (D) -- (C) -- (M) -- (N) -- cycle;

            \node[vertex] (Ln) at (L) {};
            \node[vertex] (Mn) at (M) {};
            \node[vertex] (Nn) at (N) {};

            \fill[facet] (A) -- (B) -- (C) -- (D) -- cycle;
            \fill[facet] (B) -- (C) -- (F) -- (E) -- cycle;
            \fill[facet] (C) -- (D) -- (G) -- (F) -- cycle;
            \fill[facet] (E) -- (F) -- (J) -- (I) -- cycle;
            \fill[facet] (F) -- (G) -- (K) -- (J) -- cycle;

            \node[vertex] (Dn) at (D) {}; 
            \node[vertex] (Gn) at (G) {}; 
            \node[vertex] (An) at (A) {}; 
            \node[vertex] (Bn) at (B) {}; 
            \node[vertex] (Cn) at (C) {}; 
            \node[vertex] (En) at (E) {}; 
            \node[vertex] (Fn) at (F) {}; 
            \node[vertex] (In) at (I) {}; 
            \node[vertex] (Jn) at (J) {}; 
            \node[vertex] (Kn) at (K) {};

            \draw[edgeb,back]
                (Mn) -- (Cn)
                (Nn) -- (Dn)
                (Ln) -- (Bn)
                ;

            \draw[edge,back]
                (Ln) -- (Mn)
                (Mn) -- (Nn)
                (Mn) -- (Jn)
                (An) -- (Hn)
                (Hn) -- (Ln)
                (Nn) -- (Hn)
                (Ln) -- (In)
                (Kn) -- (Nn)
                ;
            \draw[edge]
                (Cn) -- (Dn)
                (Dn) -- (An)
                (Dn) -- (Gn)
                (Fn) -- (Gn)
                (Gn) -- (Kn)
                (An) -- (Bn)
                (Bn) -- (Cn)
                (Bn) -- (En)
                (Cn) -- (Fn)
                (En) -- (Fn)
                (En) -- (In)
                (Fn) -- (Jn)
                (In) -- (Jn)
                (Jn) -- (Kn)
                ;
        \end{tikzpicture}}
        \hskip 1.5em
        \subcaptionbox[]{\(\CC(P,\lambda)\)\label{subfig:vertex-C}}[0.33\textwidth][c]{
        \begin{tikzpicture}[scale=1,baseline={(current bounding box.center)},
                x={(0:.8cm)},
                y={(230:.7cm)},
                z={(270:.8cm)},
                back/.style={dotted},
                edge/.style={thick},
                edgeb/.style={thick,color=red!50!black},
                facet/.style={fill=black!30,fill opacity=0.2},
                facetb/.style={fill=red,fill opacity=0.6},
                vertex/.style={inner sep=1pt,circle,draw=green!25!black,fill=green,anchor=base},
                vertexb/.style={inner sep=1pt,circle,draw=red!25!black,fill=red,anchor=base},
                axis/.style={->}
                ]
            \draw[axis] (0,0,0) -- (3.5,0,0) node[below] {\(x_p\)};
            \draw[axis] (0,0,0) -- (0,3.5,0) node[below left=-2pt] {\(x_q\)};
            \draw[axis] (0,0,0) -- (0,0,6.5) node[right] {\(x_r\)};

            \coordinate (A) at (0,0,0);
            \coordinate (B) at (2,0,0);
            \coordinate (C) at (2,2,0);
            \coordinate (D) at (0,2,0);
            \coordinate (E) at (3,0,0);
            \coordinate (F) at (3,3,0);
            \coordinate (G) at (0,3,0);
            \coordinate (I) at (3,0,1);
            \coordinate (J) at (3,3,1);
            \coordinate (K) at (0,3,1);
            \coordinate (L) at (2,0,2);
            \coordinate (H) at (0,0,2);
            \coordinate (M) at (2,2,2);
            \coordinate (N) at (0,2,2);

            \node[vertex] (Hn) at (H) {}; 

            \fill[facetb] (B) -- (C) -- (M) -- (L) -- cycle;
            \fill[facetb] (C) -- (F) -- (J) -- (M) -- cycle;
            \fill[facetb] (D) -- (C) -- (M) -- (N) -- cycle;

            \node[vertex] (Ln) at (L) {};

            \fill[facet] (A) -- (G) -- (K) -- (N) -- (M) -- (J) -- (I) -- (E) -- cycle;

            \node[vertex] (Mn) at (M) {};
            \node[vertex] (Nn) at (N) {};
            \node[vertexb] (Dn) at (D) {}; 
            \node[vertex] (Gn) at (G) {}; 
            \node[vertex] (An) at (A) {}; 
            \node[vertexb] (Bn) at (B) {}; 
            \node[vertexb] (Cn) at (C) {}; 
            \node[vertex] (En) at (E) {}; 
            \node[vertex] (Fn) at (F) {}; 
            \node[vertex] (In) at (I) {}; 
            \node[vertex] (Jn) at (J) {}; 
            \node[vertex] (Kn) at (K) {};

            \draw[edgeb,back]
                (Mn) -- (Cn)
                (Nn) -- (Dn)
                (Ln) -- (Bn)
                (Cn) -- (Dn)
                (Cn) -- (Fn)
                (Bn) -- (Cn)
                ;

            \draw[edge,back]
                (Ln) -- (Mn)
                (An) -- (Hn)
                (Hn) -- (Ln)
                (Nn) -- (Hn)
                (Ln) -- (In)
                ;
            \draw[edge]
                (Mn) -- (Jn)
                (Mn) -- (Nn)
                (Kn) -- (Nn)
                (Dn) -- (An)
                (Dn) -- (Gn)
                (Fn) -- (Gn)
                (Gn) -- (Kn)
                (An) -- (Bn)
                (Bn) -- (En)
                (En) -- (Fn)
                (En) -- (In)
                (Fn) -- (Jn)
                (In) -- (Jn)
                (Jn) -- (Kn)
                ;
            \node[vertex] at (F) {}; 
        \end{tikzpicture}}}
        \caption[Tropical subdivisions of poset polytopes]{Tropical subdivision of the marked poset polytope from \Cref{ex:vertices} for \(t_r=0\), \(\tfrac 1 2\) and \(1\). For generic \(t\) the vertices of the subdivision coincide with the vertices of the polytope.}
        \label{fig:genmpp-vertices}
    \end{figure}
\end{example}

To prove \Cref{thm:generic-vertices}, we first need a lemma simplifying the description of vertices in \(\TT(P,\lambda)\).
Recall that the tropical hyperplane arrangement introduced in \Cref{sec:tropical} has tropical hyperplanes enumerated by \(R\), the set of all unmarked elements in \(P\) covering at least two other elements.

\begin{lemma} \label{lem:vertex}
    Let \(v\) be a vertex in the tropical subdivision \(\TT(P,\lambda)\) of a marked order polyhedron \(\OO(P,\lambda)\) and denote by \(F\) and \(G\) the minimal faces of \(\OO(P,\lambda)\) and \(\TT(\HH(P,\lambda))\) containing \(v\), respectively, so that \(\{v\}=F\cap G\).
    Denote by \(R_G\) the set of all \(r\in R\) such that \(|\tc(v)_r|\ge 2\) and let 
    \begin{align*}
        G' &= \left\{ \, x\in\R^P \,\middle|\,
            x_q=x_{q'} \text{ for all \(r\in R_G\), \(q,q'\in\tc(v)_r\)}\,\right\}.
    \end{align*}
    Then
    \begin{equation*}
        \{v\} = F\cap G = F \cap G'.
    \end{equation*}
\end{lemma}

\begin{proof}
    By definition of the tropical subdivision \(\TT(\HH(P,\lambda))\) of \(\R^P\), we have
    \begin{equation*}
        G = \left\{ \, x\in\R^P \,\middle|\, 
            \tc(v)_r \subseteq \tc(x)_r \text{ for all \(r\in R\)} \right\},
    \end{equation*}
    where \(\tc(v)_r\subseteq \tc(x)_r\) is equivalent to \(x_q=x_{q'}\) for \(q,q'\in\tc(v)_r\) and \(x_{q''}\le x_q\) for \(q''\prec r\) with \(q''\notin\tc(v)_r\) and \(q\in\tc(v)_r\).
    Hence, we may write
    \begin{equation*}
        G = G' \cap H \cap L,
    \end{equation*}
    where
    \begin{align*}
        G' &= \left\{ \, x\in\R^P \,\middle|\,
            x_q=x_{q'} \text{ for all \(r\in R_G\), \(q,q'\in\tc(v)_r\)}\,\right\}, \\
        H &= \left\{ \, x\in\R^P \,\middle|\,
            x_{q''}\le x_q \text{ for all \(r\in R_G\), \(q''\prec r\) with \(q''\notin\tc(v)_r\), \(q\in\tc(v)_r\)} \,\right\}, \\
        L &= \left\{ \, x\in\R^P \,\middle|\,
            \tc(v)_r \subseteq \tc(x)_r \text{ for all \(r\in R\setminus R_G\)}\,\right\}.
    \end{align*}
    Since \(v_{q''}<v_q\) for \(q'',q\prec r\) with \(q''\notin\tc(v)_r\) and \(q\in\tc(v)_r\), we know that \(v\) is an interior point of \(H\).
    Since for \(r\notin R_G\) the set \(\tc(v)_r\) has exactly one element, there are no conditions \(x_q=x_{q'}\) for \(q,q'\in\tc(v)_r\) and by the previous argument \(v\) is also an interior point of \(L\).
    Hence, we have
    \begin{equation*}
        \{v\} = (F \cap G') \cap (H\cap L),
    \end{equation*}
    where \(v\) is an interior point of \(H\cap L\).
    Since \(\R^P\) is Hausdorff and \(F\cap G'\) is connected, this implies \(\{v\} = F\cap G'\).
\end{proof}

We are now ready to prove \Cref{thm:generic-vertices}. As a reference on face partitions of marked posets as used in the following proof we refer to \cite[Section~3]{Peg17}.

\begin{proof}[{Proof of \Cref{thm:generic-vertices}}]
    Let \(v\) be a vertex in the tropical subdivision \(\TT(P,\lambda)\) of \(\OO(P,\lambda)\), so that \(\{v\}=F\cap G\), where \(F\) is the minimal face of \(\OO(P,\lambda)\) containing \(v\) and \(G\) is the minimal cell in \(\TT(\HH(P,\lambda))\) containing \(v\).
    Let \(\tc(G)\) be the tropical covector corresponding to \(G\) and denote by \(R_G\) the set of all \(r\in R\) such that \(|\tc(G)_r|\ge 2\).
    In other words, \(R_G\) consists of all \(p\in\tilde P\) such that at least two different \(q\prec p\) maximize \(v_q\).
    Fix \(u\in\smash{[0,1]^{\tilde P}}\) with \(u_p\in (0,1)\) for \(p\in R_G\) and \(u_p=0\) otherwise.

    We claim that \(\varphi_u(v)\) is a vertex of \(\OO_u(P,\lambda)\).
    Since \(u\) is a degeneration of any \(t\in\smash{(0,1)^{\tilde P}}\), we can conclude by \Cref{prop:degeneration} that \(\varphi_t(v)\) is then also a vertex of \(\OO_t(P,\lambda)\) whenever \(t\in\smash{(0,1)^{\tilde P}}\).

    By \Cref{lem:vertex} we have \(\{v\}=F\cap G'\), where \(G'\) is defined by the conditions \(x_q=x_{q'}\) for \(r\in R_G\), \(q,q'\in\tc(v)_r\). 
    Let \(Q\) be the minimal face of \(\OO_u(P,\lambda)\) containing \(\varphi_u(v)\).
    If we can show \(\psi_u(Q)\subseteq F\) and \(\psi_u(Q)\subseteq G'\),  we can conclude that \(Q\) is a single point and hence \(\varphi_u(v)\) a vertex.
    Since \(0\in\smash{[0,1]^{\tilde P}}\) is a degeneration of \(u\), we have \(\psi_u(\relint Q)\subseteq \relint F\) by \Cref{prop:degeneration} and conclude \(\psi_u(Q)\subseteq F\) by taking closures.

    To show that \(\psi_u(Q)\subseteq G'\), let \(r\in R_G\) and \(q\in\tc(G)_r=\tc(v)_r\).
    For \(y\in Q\) with image \(z=\psi_u(y)\) in \(\OO(P,\lambda)\), we will show \(q\dashv_z r\), so \(q\in\tc(z)_r\).
    Hence, we obtain \(\tc(v)_r\subseteq\tc(z)_r\) for \(r\in R_G\) which implies \(z\in G'\).
    Our strategy is as follows: construct a chain \(\Fc\) corresponding to a defining inequality of \(\OO_u(P,\lambda)\) satisfied by \(\varphi_u(v)\) with equality, such that \(q \dashv_v r\) is one of the corresponding conditions on \(v\) in \Cref{prop:ineq-pullback}.
    Since the inequality is satisfied by \(\varphi_u(v)\) with equality, the same holds for \(y\in Q\).
    Again, by \Cref{prop:ineq-pullback}, this implies that \(q \dashv_z r\).

    What remains to be done is constructing the chain \(\Fc\).
    In the following, we need a relation slightly stronger than \(\dashv_x\).
    Let \(a\Dashv_x b\) denote the relation on \(P\) defined by \(b\in R_G\) and \(a\dashv_x b\).
    That is, \(a\Dashv_x b\) holds if and only if \(b\in R_G\), \(a\prec b\) and \(x_a=\max_{q\prec b} x_q\).

    First construct a chain from \(q\) downward to a marked element that is of the kind
    \begin{equation*}
        a \prec \cdots \prec p' \prec p'_1 \Dashv_v \cdots \Dashv_v p'_l = q,
    \end{equation*}
    where \(l\ge 1\) and \(p'_1\notin R_G\).
    That is, walk downwards in \(R_G\) along relations \(\Dashv_v\) as long as possible, then arbitrarily extend the chain to some marked element \(a\in P^*\).
    Let
    \begin{equation*}
        \Fc\colon a\prec \cdots \prec p' \prec p'_1\prec \cdots \prec p'_{l-1}.
    \end{equation*}
    When \(q\) was marked, \(\Fc\) is just the empty chain. When \(q\notin R_G\), we have \(l=1\), \(p'_1=q\) and \(\Fc\) ends in \(p'\).

    Now construct a maximal chain
    \begin{equation*}
        q \Dashv_v r \Dashv_v p_1 \Dashv_v \cdots \Dashv_v p_k,
    \end{equation*}
    where \(k\ge 0\).
    Let \(p_{-1}=q\), \(p_0=r\), and let \(B\in\pi_F\) be the block of the face partition of \(F\) containing \(p_k\).
    We claim that \(B\) can not be a singleton: since \(F\cap G'\) is a point, the conditions imposed by the face partition \(\pi_F\) together with the conditions given by \(G'\) determine all the coordinates, in particular \(x_{p_k}\). However, \(p_k\) is neither marked, since it is an element of \(R_G\subseteq\tilde P\), nor does it appear in one of the equations for \(G'\), since the chain was chosen maximal.
    Hence, the coordinate \(x_{p_k}\) must be determined by \(p_k\) sitting in a non-trivial block with some other coordinate already determined by the conditions imposed by \(\lambda\), \(\pi_F\) and \(G'\).

    If there exists \(p\in B\) with \(p_k\prec p\), let
    \begin{equation*}
        \Fd\colon p_1 \prec \cdots \prec p_k \prec p.
    \end{equation*}
    The chain \(\Fc\prec q\prec r\prec\Fd\) yields a defining inequality for \(\OO_u(P,\lambda)\).
    Since \(u_{p'_1}=0\), while \(u_{p'_2},\dots,u_{p'_l},u_r,u_{p_1},\dots,u_{p_k}>0\) and \(u_p\neq 1\) in case of \(p\in \tilde P\), the describing inequality of \(\OO_u(P,\lambda)\) given by \(\Fc\prec q\prec r\prec\Fd\) is satisfied with equality for some \(\varphi_u(x)\) if and only if
    \begin{align*}
        p'_1 \dashv_x \cdots \dashv_x p'_l = q \dashv_x r \dashv_x p_1 \dashv_x \cdots \dashv_x p_k \quad\text{and}\quad x_{p_k} = x_p.
    \end{align*}
    For \(x=v\), all these conditions are satisfied.
    Hence, they are also satisfied by \(z\).
    In particular \(q\dashv_z r\) as desired.

    If there exists no \(p\in B\) with \(p_k\prec p\), there must be some \(p\in B\) with \(p\prec p_k\), since \(B\) is not a singleton.
    In this case \(v_p=v_{p_k}\) so in particular \(p\dashv_v p_k\).
    Since \(p_{k-1}\dashv_v p_k\) as well, we conclude \(v_{p_{k-1}} = v_p = v_{p_k}\).
    Now let
    \begin{equation*}
        \Fd\colon p_1 \prec \cdots \prec p_k.
    \end{equation*}
    The inequality for \(\OO_u(P,\lambda)\) given by \(\Fc\prec q\prec r\prec\Fd\) is satisfied with equality for \(\varphi_u(x)\) if and only if
    \begin{align*}
        p'_1 \dashv_x \cdots \dashv_x p'_l = q \dashv_x r \dashv_x p_1 \dashv_x \cdots \dashv_x p_{k-1} \quad\text{and}\quad x_{p_{k-1}} = x_{p_k}.
    \end{align*}
    Again, all these conditions hold for \(x=v\), hence also for \(z\) and we can conclude \(q \dashv_z r\) as before.
\end{proof}

We finish this section by the following conjecture.

\begin{conjecture}
For any vertex $v$ in the generic marked poset polytope, there exists a vertex $t_v$ of the hypercube $[0,1]^{\tilde{P}}$ such that the image of $v$ under the degeneration map is a vertex in $\OO_{t_v}(P,\lambda)$. 
\end{conjecture}

\section{Poset Transformations}
\label{sec:poset-transforms}

Since having a strict or even regular marking already played an essential role in the theory of marked order polyhedra, it is a natural question to ask whether we can apply the poset transformation used in \cite{Peg17} and still obtain the same marked poset polyhedra up to affine equivalence for arbitrary \(t\in\smash{[0,1]^{\tilde P}}\).
In this section we show that the answer is positive: modifying a marked poset to be strictly marked and modifying a strictly marked poset to be regular does not change the affine isomorphism type.

Recall from \cite{Peg17} that a \emph{constant interval} in a marked poset $(P,\lambda)$ is an interval $[a,b]$ such that $a,b\in P^*$ are marked with $\lambda(a)=\lambda(b)$. The consequence is that for any point $x$ in the associated marked order polyhedron $\OO(P,\lambda)$ we have $x_p=\lambda(a)=\lambda(b)$ whenever $p\in[a,b]$.

\begin{proposition}
    Contracting constant intervals in \((P,\lambda)\) yields a strictly marked poset \((P/\pi,\lambda/\pi)\) such that \(\OO_{t'}(P/\pi,\lambda/\pi)\) is affinely isomorphic to \(\OO_t(P,\lambda)\) for all \(t\in\smash{[0,1]^{\tilde P}}\), where $t'$ is the restriction of $t$ to elements not contained in any non-trivial constant intervals.
\end{proposition}

\begin{proof}
    Let \((P',\lambda')\) be the strictly marked poset obtained from \((P,\lambda)\) by contracting constant intervals.
    Hence, \(P'\) is obtained from \(P\) by taking the quotient under the equivalence relation generated by \(a\sim p\) and \(p\sim b\) whenever \(a\le b\) are marked elements such that \(\lambda(a)=\lambda(b)\) and \(a\le p\le b\).
    The elements of $P'$ are either singletons $\{p\}$ for $p$ not contained in any non-trivial constant interval or non-trivial blocks $B$ that are unions of non-trivial constant intervals.
    All non-trivial blocks $B$ are marked and among the singletons $\{p\}$ only those with $p\in \tilde P$ are unmarked.

    By \cite[Proposition 3.18]{Peg17}, we have an \(q^*\colon \OO(P',\lambda')\to\OO(P,\lambda)\) induced by the quotient map \(q\colon (P,\lambda)\to(P',\lambda')\).
    Now consider the two transfer maps \(\varphi_t\colon\OO(P,\lambda)\to\OO_t(P,\lambda)\) and \(\varphi'_{t'}\colon\OO(P',\lambda')\to\OO_{t'}(P',\lambda')\).
    When \(B\) is a non-trivial block in \(P'\)---in other words an equivalence class with at least two elements---we have
    \begin{equation*}
        \varphi_t(x)_p = (1-t_p)\lambda'(B)
    \end{equation*}
    for all unmarked \(p\in B\) and \(x\in\OO(P,\lambda)\).
    When \(p\) is an unmarked element outside of constant intervals, we have \(\varphi_t(q^*(x))_p = \varphi'_{t'}(x)_{\{p\}}\) for all $x\in\OO(P',\lambda')$ by construction.
    Hence, the affine map \(\gamma\colon \R^{\tilde P'}\to\R^{\tilde P}\) defined by
    \begin{equation*}
        \gamma(x)_p = 
        \begin{cases}
            (1-t_p)\lambda'(B) &\text{if \(p\in \tilde P\cap B\) for a non-trivial block \(B\),} \\
            x_{\{p\}} &\text{otherwise} \\
        \end{cases}
    \end{equation*}
    restricts to an affine map \(\OO_{t'}(P',\lambda')\to\OO_t(P,\lambda)\), such that the diagram
    \begin{equation*}
        \begin{tikzcd}
            \OO(P',\lambda') \arrow[r,"q^*"] \arrow[d, "\varphi'_{t'}"] & \OO(P,\lambda) \arrow[d,"\varphi_t"]\\
            \OO_{t'}(P',\lambda') \arrow[r,"\gamma"] & \OO_t(P,\lambda)
        \end{tikzcd}
    \end{equation*}
    commutes.
    Thus, it is an affine isomorphism.
    Note that we used the projected polyhedra and transfer maps in the above diagram.
\end{proof}

Recall from \cite{Peg17} that a covering relation $p\prec q$ in $(P,\lambda)$ is called \emph{non-redundant} if for all marked elements $a,b\in P^*$ with with $a\le q$ and $p\le b$, we have $a=b$ or $\lambda(a)<\lambda(b)$. A marked poset is called \emph{regular} if all its covering relations are non-redundant.

\begin{proposition}
    If \((P,\lambda)\) is strictly marked, removing a redundant covering relation yields a marked poset \((P',\lambda)\) such that \(\OO_t(P,\lambda)=\OO_t(P',\lambda)\) for all \(t\in\smash{[0,1]^{\tilde P}}\).
\end{proposition}

\begin{proof}
    Let \(p\prec q\) be a redundant covering relation in \(P\).
    That is, there are marked elements \(a\neq b\) satisfying \(a\le q\), \(p\le b\) and \(\lambda(a)\ge \lambda(b)\).
    Let \(P'\) be obtained from \(P\) be removing the covering relation \(p\prec q\).

    Comparing the transfer maps \(\varphi_t\) and \(\varphi_t'\) associated to \((P,\lambda)\) and \((P',\lambda)\) defined on the same marked order polyhedron \(\OO(P,\lambda)=\OO(P',\lambda)\) by \cite[Proposition 3.24]{Peg17}, we see that they can only differ in the \(q\)-coordinate, which can only happen when \(q\) is unmarked.
    To be precise,
    \begin{equation*}
        \varphi_t(x)_q = x_q - t_q \max_{q'\prec q} x_{q'} \qquad\text{and}\qquad
        \varphi_t'(x)_q = x_q - t_q \max_{\substack{q'\prec q\\q'\neq p}} x_{q'}.
    \end{equation*}

    Since \(\lambda\) is strict, we can not have \(a\le p\).
    Otherwise we had \(a<b\) in contradiction to \(\lambda(a)\ge \lambda(b)\). Hence, when \(q\) is unmarked, there is a \(p'\neq p\) such that \(a\le p'\prec q\).
    For all \(x\in\OO(P',\lambda)=\OO(P,\lambda)\) we have
    \begin{equation*}
        x_{p'} \ge \lambda(a) \ge \lambda(b) \ge x_p
    \end{equation*}
    and excluding \(p\) from the maximum does not change the transfer map at all. We conclude that
    \begin{equation*}
        \OO_t(P',\lambda) = \varphi_t'(\OO(P',\lambda)) = \varphi_t(\OO(P,\lambda)) = \OO_t(P,\lambda). \qedhere
    \end{equation*}
\end{proof}

Using the above transformations, we can always replace a marked poset \((P,\lambda)\) by a regular marked poset \((P',\lambda')\) yielding affinely equivalent marked poset polyhedra.

\begin{remark}
    If $(P,\lambda)$ is integrally marked and $t\in\{0,1\}^{\tilde P}$, the above constructions actually yield unimodular isomorphisms.
\end{remark}

\section{Facets and the Hibi--Li Conjecture}\label{sec:facetHL}

In \cite{Peg17} it is proved that regular marked posets yield a one-to-one correspondence of covering relations in \((P,\lambda)\) and facets of \(\OO(P,\lambda)\).
We strongly believe the same regularity condition implies that both the inequalities in \Cref{def:mpp} for \(t\in\smash{(0,1)^{\tilde P}}\) and the inequalities in \Cref{prop:mcop} for all partitions \(\tilde{P}=C\sqcup O\)---i.e., all \(t\in\smash{\{0,1\}^{\tilde P}}\)---correspond to the facets of the described polyhedra.
In fact, we can show that the latter implies the former and the conjecture is true for certain \emph{ranked} marked posets.

\begin{definition} \label{def:tame}
    A marked poset \((P,\lambda)\) is called \emph{tame} if the inequalities given in \Cref{prop:mcop} correspond to the facets of \(\OO_{C,O}(P,\lambda)\) for all partitions \(\tilde P=C\sqcup O\).
\end{definition}

\begingroup
\makeatletter
\apptocmd{\theconjecture}{\unless\ifx\protect\@unexpandable@protect\protect\footnotemark\fi}{}{}
\makeatother
\begin{conjecture} \label{conj:regular-tame}
    \footnotetext{In \cite[Proposition~4.5]{FF16}, the statement of \Cref{conj:regular-tame} is given without proof for the case of admissible partitions and bounded polyhedra.}
    A marked poset \((P,\lambda)\) is tame if and only if it is regular.
\end{conjecture}
\endgroup

We know that regularity is a necessary condition for being tame, since otherwise \((P,\lambda)\) either contains non-trivial constant intervals and the covering relations in those do not correspond to facets of \(\OO(P,\lambda)\) or the marking is strict but there are redundant covering relations that do not correspond to facets of \(\OO(P,\lambda)\).

We start by considering marked chain polyhedra.
We can show that any chain in \((P,\lambda)\) that does not contain redundant covering relations defines a facet of \(\CC(P,\lambda)\).

\begin{lemma} \label{lem:chain}
    Let \((P,\lambda)\) be a marked poset and \(\Fc\colon a \prec p_1 \prec p_2 \prec \cdots \prec p_r \prec b\) be a saturated chain between elements \(a,b\in P^*\) with all \(p_i\in \tilde P\) and \(r\ge 1\).
    If none of the covering relations in \(\Fc\) are redundant, the inequality
    \begin{equation} \label{eq:chaineq-goal}
        x_{p_1} + \cdots + x_{p_r} \le x_b - x_a
    \end{equation}
    is not redundant in the description of \(\CC(P,\lambda)=\OO_{\tilde P,\varnothing}(P,\lambda)\) given in \Cref{prop:mcop}.
\end{lemma}

In particular we obtain the following result.

\begingroup
\makeatletter
\apptocmd{\thecorollary}{\unless\ifx\protect\@unexpandable@protect\protect\footnotemark\fi}{}{}
\makeatother
\begin{corollary}\label{cor:chain-polyhedron}
    \footnotetext{The result of \Cref{cor:chain-polyhedron} was previously stated without proof in \cite[Lemma~1]{Fou16} for bounded polyhedra.}
    Let \((P,\lambda)\) be regular. The description of the marked chain polyhedron \(\CC(P,\lambda)=\OO_{\tilde P,\varnothing}(P,\lambda)\) given in \Cref{prop:mcop} is non-redundant.
\end{corollary}
\endgroup

\begin{proof}[Proof of \Cref{lem:chain}]
    Our strategy is as follows.
    First show that \eqref{eq:chaineq-goal} can be strictly satisfied by some point in \(\CC(P,\lambda)\), so the polyhedron is not contained in the corresponding hyperplane.
    Then construct a point \(x\in\CC(P,\lambda)\) such that \eqref{eq:chaineq-goal} is satisfied with equality but all other inequalities that can be strictly satisfied by points in \(\CC(P,\lambda)\) are strictly satisfied by \(x\).
    This shows that \eqref{eq:chaineq-goal} is the only inequality describing a facet with \(x\) in its relative interior.

    To see that \eqref{eq:chaineq-goal} can be strictly satisfied, just take \(x\in\R^P\) with \(x_a=\lambda(a)\) for \(a\in P^*\) and \(x_p=0\) for \(p\in\tilde P\).
    Note that \(\lambda(a)<\lambda(b)\) since otherwise all covering relations in \(\Fc\) would be redundant.

    For the second step, first linearly order the set of all markings in \([\lambda(a), \lambda(b)]\), so that
    \begin{equation*}
        \lambda(P^*) \cap [\lambda(a),\lambda(b)] = \{ \lambda_1, \dots, \lambda_k \}
    \end{equation*}
    with \(\lambda(a) = \lambda_1 < \cdots < \lambda_k = \lambda(b)\) and \(k>1\).
    For \(i=1,\dots,k-1\) we define the following sets:
    \begin{align*}
        Z^\uparrow_i &= \left\{ \, p\in\Fc : \text{\(p \ge d\) for some \(d\in P^*\) with \(\lambda(d) \ge \lambda_{i+1}\)} \,\right\}, \\
        Z^\downarrow_i &= \left\{ \, p\in\Fc : \text{\(p \le e\) for some \(e\in P^*\) with \(\lambda(e) \le \lambda_i\)} \,\right\}, \\
        Z_i &= \Fc\setminus \left( Z^\uparrow_i \sqcup Z^\downarrow_i \right).
    \end{align*}
    Note that \(Z^\uparrow_i\) and \(Z^\downarrow_i\) are disjoint, since any \(p\) in their intersection would give \(d\le p\le e\) with \(\lambda(d)\ge\lambda_{i+1} > \lambda_i \ge \lambda(e)\) contradicting \(\lambda\) being order-preserving.

    For \(p\in Z_i^\uparrow\) all elements of \(\Fc\) greater than \(p\) are also contained in \(Z_i^\uparrow\) and for \(p\in Z_i^\downarrow\) all elements of \(\Fc\) less than \(p\) are also contained in \(Z_i^\downarrow\).
    Furthermore, we have \(a\in Z_i^\downarrow\) and \(b\in Z_i^\uparrow\) for all \(i\).
    Thus, the chain \(\Fc\) decomposes into three connected subchains \(Z_i^\downarrow\), \(Z_i\), \(Z_i^\uparrow\).

    We claim that the middle part \(Z_i\) is always non-empty as well.
    Otherwise, the chain \(\Fc\) contains a covering relation \(p\prec q\) with \(p\in Z_i^\downarrow\) and \(q\in Z_i^\uparrow\) and hence we had \(d,e\in P^*\) with \(e\ge p\prec q \ge d\) and \(\lambda(e)\le \lambda_i < \lambda_{i+1} \le \lambda(d)\) so that \(p\prec q\) is redundant.

    We also claim that each \(p_j\in\Fc\) is contained in at least one of the \(Z_i\).
    Since \(a\le p_j\), we can choose \(i_0\in[k]\) maximal such that \(p_j\ge d\) for some \(d\) with \(\lambda(d)\ge \lambda_{i_0}\).
    In the same fashion, choose \(i_1\in[k]\) minimal such that \(p_j\le e\) for some \(e\) with \(\lambda(e)\le \lambda_{i_1}\).
    We have \(i_0<i_1\), since otherwise there are \(d\le p_j\le e\) with \(\lambda(d)\ge\lambda(e)\), either rendering \(\lambda\) non order-preserving or any covering relation above or below \(p_j\) redundant.
    We conclude that \(p_j\in Z_i\) for \(i=i_0,\dots,i_1-1\).

    Define a point \(x\in\R^P\) by letting \(x_a=\lambda(a)\) for all \(a\in P^*\) and for \(p\in\tilde P\):
    \begin{equation*}
        x_p =
        \begin{cases} \displaystyle
            \smashoperator[r]{\sum_{\substack{i=1,\dots,k-1,\\p\in Z_i}}} \frac{\lambda_{i+1} - \lambda_i}{|Z_i|} & \text{for \(p\in\Fc\) and} \\
            \quad\varepsilon &\text{for \(p\in\tilde P\setminus\Fc\),}
        \end{cases}
    \end{equation*}
    where \(\varepsilon>0\) is small enough to satisfy the finitely many constraints in the rest of this proof.
    Note that all \(|Z_i|>0\) since the \(Z_i\) are non-empty and all \(x_p>0\) for \(p\in\tilde P\) since each \(p_j\in\Fc\) is contained in at least one of the \(Z_i\).

    The inequality given by \(\Fc\) is satisfied with equality, since
    \begin{align*}
        \sum_{j=1}^r x_{p_j} &=
        \sum_{j=1}^r \sum_{\substack{i=1,\dots,k-1,\\p_j\in Z_i}} \frac{\lambda_{i+1} - \lambda_i}{|Z_i|} \\
        &= \sum_{i=1}^{k-1} \sum_{\substack{j=1,\dots,r\\p_j\in Z_i}} \frac{\lambda_{i+1} - \lambda_i}{|Z_i|}
        = \sum_{i=1}^{k-1} (\lambda_{i+1} - \lambda_i)
        = \lambda_k - \lambda_1 = x_b - x_a.
    \end{align*}

    Now consider any chain \(\Fd\colon a' \prec q_1 \prec \cdots \prec q_s \prec b'\) different from \(\Fc\).
    We have to show that the inequality
    \begin{equation} \label{eq:second-chain}
        x_{q_1} + \cdots + x_{q_s} \le x_{b'} - x_{a'}
    \end{equation}
    either can not be strictly satisfied by any point in \(\CC(P,\lambda)\) or is strictly satisfied by \(x\).

    If \(\lambda(a')=\lambda(b')\) the inequality can never be satisfied strictly by points in \(\CC(P,\lambda)\).
    If \(\lambda(a')<\lambda(b')\) we have
    \begin{equation} \label{eq:some-eps}
        \sum_{j=1}^s x_{q_j} = \sum_{q\in\tilde\Fd} x_q =
            \varepsilon \left|\,\tilde \Fd\setminus\tilde \Fc\,\right| + \sum_{q\in \tilde \Fd\cap\tilde \Fc} \smashoperator[r]{\sum_{\substack{i=1,\dots,k-1,\\q\in Z_i}}} \frac{\lambda_{i+1} - \lambda_i}{|Z_i|},
    \end{equation}
    where \(\tilde\Fc\) and \(\tilde\Fd\) denote the unmarked parts of \(\Fc\) and \(\Fd\), respectively.

    Let \(S\) denote the double sum in \eqref{eq:some-eps} and consider the following cases:
    \begin{enumerate}
        \item \label{it:longcase} We have \(\lambda(a) < \lambda(a') < \lambda(b') < \lambda(b)\).
            Let \(1<i_0<i_1<k\) be the indices such that \(\lambda_{i_0}=\lambda(a')\), \(\lambda_{i_1}=\lambda(b')\).
            Note that all elements of \(\tilde\Fd\) are above \(a'\) with \(\lambda(a')=\lambda_{i_0}\) so \(\tilde\Fd\cap\tilde\Fc\subseteq Z_{i_0-1}^\uparrow\) and we have \(\tilde\Fd\cap Z_i=\varnothing\) for \(i<i_0\).
            By the same reasoning \(\tilde\Fd\cap Z_i=\varnothing\) for \(i\ge i_1\).
            Hence, we have
            \begin{equation*}
                S = \sum_{q\in \tilde \Fd\cap\tilde \Fc} \smashoperator[r]{\sum_{\substack{i=i_0,\dots,i_1-1,\\q\in Z_i}}} \frac{\lambda_{i+1} - \lambda_i}{|Z_i|} \le \lambda(b')-\lambda(a'),
            \end{equation*}
            with equality achieved if and only if \(Z_i\subseteq\tilde\Fd\) for \(i=i_0,\dots,i_1-1\).
            
            Let \(p\in\tilde\Fc\) be maximal such that \(p\in Z_{i_1-1}\).
            Then there is a covering relation \(p\prec q\) in \(\Fc\) with \(q\in Z_{i_1-1}^\uparrow\).
            We have \(p\notin\Fd\), since otherwise \(p<b'\) and \(q>d\) for some \(d\) with \(\lambda(d)\ge \lambda_{i_1} = \lambda(b)\), rendering \(p\prec q\) redundant.
            Hence, \(p\in Z_{i_1-1}\setminus\tilde\Fd\) and \(S<\lambda(b')-\lambda(a')\).

            We conclude that \eqref{eq:second-chain} is strictly satisfied for small enough \(\varepsilon\).
        \item We have \(\lambda(a')< \lambda(b') \le \lambda(a)\) or \(\lambda(b) \le \lambda(a') < \lambda(b')\).
            In this case we have \(\Fd\subseteq Z_i^\downarrow\) for all \(i\) or \(\Fd\subseteq Z_i^\uparrow\) for all \(i\), respectively, so that \(S=0\).
            Choosing \(\varepsilon\) small enough yields strict inequality in \eqref{eq:second-chain}.

        \item We have \(\lambda(a)<\lambda(a')<\lambda(b)\le\lambda(b')\) or \(\lambda(a')\le\lambda(a)<\lambda(b')<\lambda(b)\).
            By reasoning similar to \cref{it:longcase} we have \(S<\lambda(b)-\lambda(a')\) or \(S<\lambda(b')-\lambda(a)\), respectively.
            In both cases \(S<\lambda(b')-\lambda(a')\) and choosing \(\varepsilon\) small enough yields strict inequality in \eqref{eq:second-chain}.
        \item We have \(\lambda(a')\le\lambda(a)<\lambda(b)\le\lambda(b')\). In case \(\tilde\Fd\cap\tilde\Fc=\tilde\Fc\) we have \(\lambda(a')<\lambda(a)\) and \(\lambda(b)<\lambda(b')\) since otherwise the covering relation \(a\prec p_1\) or \(p_k\prec b\) would be redundant.
            Hence
            \begin{equation*}
                \sum_{q\in\tilde\Fd} x_q =
                    \varepsilon \left|\,\tilde \Fd\setminus\tilde \Fc\,\right| + \left(\lambda(b)-\lambda(a)\right) < \lambda(b')-\lambda(a')
            \end{equation*}
            for \(\varepsilon\) small enough.

            In case \(\tilde\Fd\cap\tilde\Fc\neq\tilde\Fc\), at least one summand is missing in \(S\) to achieve \(\lambda(b)-\lambda(a)\) since each \(p\in\tilde\Fc\) is in at least one of the \(Z_i\).
            Thus, \(S<\lambda(b)-\lambda(a)\) and we may choose \(\varepsilon\) small enough to obtain 
            \begin{equation*}
                \sum_{q\in\tilde\Fd} x_q < 
                    \lambda(b)-\lambda(a) \le \lambda(b')-\lambda(a').
            \end{equation*}
    \end{enumerate}
    In all cases \eqref{eq:second-chain} is satisfied by \(x\) with strict inequality and we conclude that \eqref{eq:chaineq-goal} is not redundant in the description of \(\CC(P,\lambda)\) given in \Cref{prop:mcop}.
\end{proof}

For ranked marked posets, we can use \Cref{lem:chain} to show that \Cref{conj:regular-tame} holds.

\begin{definition}
    A marked poset \((P,\lambda)\) is called \emph{ranked} if there exists a \emph{rank function} \(\rk\colon P\to \Z\) satisfying
    \begin{enumerate}
           \item \(\rk p + 1 = \rk q\) for all \(p,q\in P\) with \(p\prec q\),
        \item \(\lambda(a) < \lambda(b)\) for all \(a,b\in P^*\) with \(\rk a < \rk b\).
    \end{enumerate}
\end{definition}

Note that the rank function of a ranked marked poset is uniquely determined up to a constant on each connected component. 

\begin{proposition} \label{prop:ranked-tame}
    Let \((P,\lambda)\) be regular and ranked, then \((P,\lambda)\) is tame.
\end{proposition}

\begin{proof}
    Let \(\rk\colon P\to\Z\) be a rank function such that \(\min\{\rk p : p\in P\}=0\) and let \(r=\max\{\rk p: p\in P\}\).
    Since \(\lambda(a)<\lambda(b)\) for marked elements with \(\rk a<\rk b\), we can choose real numbers \(\xi_0<\xi_1 < \cdots < \xi_{r+1}\) such that \(\lambda(a) \in (\xi_i, \xi_{i+1})\) for \(a\in P^*\) with \(\rk a=i\).

    Let \(\tilde P=C\sqcup O\) be any partition.
    All inequalities \(0\le x_p\) for \(p\in C\) are non-redundant in the description of \(\OO_{C,O}(P,\lambda)\) given in \Cref{prop:mcop}.
    To see this, take any \(x\in\OO_{C,O}(P,\lambda)\) and let \(x'\in\R^P\) be given by \(x'_q=x_q\) for \(q\neq p\) and \(x_p=-1\).

    Now consider any chain \(\Fc\colon a\prec p_1\prec\cdots\prec p_r\prec b\) with \(a,b\in P^*\sqcup O\) and all \(p_i\in\tilde P\).
    If \(r=0\), we have to show that \(x_a\le x_b\) is a non-redundant inequality provided at least one of \(a\) and \(b\) is not marked.
    For this, define \(x\in\OO_{C,O}(P,\lambda)\) by
    \begin{equation*}
        x_p =
        \begin{cases}
            \lambda(p) &\text{for \(p\in P^*\),} \\
            \xi_{\rk p} &\text{for \(p\in O\setminus\{a,b\}\) with \(\rk p\le\rk a\),} \\
            \xi_{\rk p+1} &\text{for \(p\in O\setminus\{a,b\}\) with \(\rk p\ge\rk b\),} \\
            \xi_{\rk b} &\text{for \(p\in\{a,b\}\) if \(a,b\in O\),} \\
            \lambda(a) &\text{for \(p\in\{a,b\}\) if \(a\notin O\),} \\
            \lambda(b) &\text{for \(p\in\{a,b\}\) if \(b\notin O\),} \\
            \min_i \{\xi_{i+1}-\xi_i\} & \text{for \(p\in C\).}
        \end{cases}
    \end{equation*}
    Using the fact that \((P,\lambda)\) is ranked it is routine to check that \(x\) satisfies all inequalities of \Cref{prop:mcop} strictly except for \(x_a\le x_b\).

    Now consider the case where \(r\ge 1\).
    The idea is to extend the marking \(\lambda\) to a marking \(\lambda'\) defined on \(P^*\sqcup O\) such that \((P,\lambda')\) has no redundant covering relations in \(\Fc\).
    We then have \(\OO_{C,O}(P,\lambda)\cap U = \CC(P,\lambda')\) with \(U\) given by \(x_p = \lambda'(p)\) for \(p\in O\).
    Note that the description of \(\CC(P,\lambda')\) in \Cref{prop:mcop} is exactly the description given for \(\OO_{C,O}(P,\lambda)\) in \Cref{prop:mcop} with the additional equations \(x_p = \lambda'(p)\) for \(p\in O\).
    In the description of \(\CC(P,\lambda')\) the inequality given by \(\Fc\) is not redundant by \Cref{lem:chain} and hence the same inequality is not redundant in the description of \(\OO_{C,O}(P,\lambda)\cap U\).
    Thus, it can not be redundant in the description of \(\OO_{C,O}(P,\lambda)\) itself either.

    It remains to construct the extended marking \(\lambda'\). Let \(\lambda'(p)=\lambda(p)\) for \(p\in P^*\) and for \(p\in O\) with \(\rk p=i\) choose
    \begin{equation*}
        \lambda'(p) \in
        \begin{cases}
            \left(\xi_i, \xi_{i+1}\right) &\text{for \(i\notin\{\rk a, \rk b\}\),} \\
            \left( \max\left\{\xi_i, \max\{\lambda(d) : \rk d=i\}\right\}, \xi_{i+1} \right) &\text{for \(p=a\) if \(a\in O\),} \\
            \left( \xi_i, \min\left\{\xi_{i+1}, \min\{\lambda(d) : \rk d=i\}\right\}\right) &\text{for \(p=b\) if \(b\in O\),} \\
            \left( \xi_i, \lambda'(a) \right) &\text{for \(i=\rk a\), \(p\neq a\),} \\
            \left( \lambda'(b), \xi_{i+1} \right) &\text{for \(i=\rk b\), \(p\neq b\).}
        \end{cases}
    \end{equation*}
    The appearing open intervals are all non-empty so these choices are possible.
    Given any such \(\lambda'\), we still have \(\lambda'(d)\in (\xi_i,\xi_{i+1})\) when \(\rk d=i\), so \((P,\lambda')\) is still a ranked marked poset.
    Let us verify that \(\Fc\) contains no redundant covering relation with respect to \((P,\lambda')\).
    \begin{enumerate}
        \item The covering relation \(a\prec p_1\) is non-redundant since \(\lambda'(d)<\lambda'(a)\) for all marked elements \(d\le p_1\), \(d\neq a\).
        \item The covering relation \(p_r\prec b\) is non-redundant since \(\lambda'(d)>\lambda'(b)\) for all marked elements \(d\ge p_r\), \(d\neq b\).
        \item All covering relations \(p_j\prec p_{j+1}\) are non-redundant since \((P,\lambda)\) is ranked.
    \end{enumerate}
    Hence, we can apply \Cref{lem:chain} to \(\CC(P,\lambda')\) and obtain the desired result.
\end{proof}

\begin{remark}
    In light of the proof of \Cref{prop:ranked-tame}, a possible strategy to prove \Cref{conj:regular-tame} in general would be to extend markings such that along a given chain the covering relations stay non-redundant.
    However, we did not succeed in doing this for arbitrary (non-ranked) marked posets.
\end{remark}

\begin{remark}
    The marked posets relevant in representation theory appearing in \cite{ABS11,BD15} are all ranked and regular after applying the transformations of \Cref{sec:poset-transforms} if necessary.
    Hence, they are tame and \Cref{prop:mcop} gives non-redundant descriptions for all associated marked chain-order polyhedra.
\end{remark}

We show that \((P,\lambda)\) being tame also implies the description given for generic marked poset polyhedra \(\OO_t(P,\lambda)\) in \Cref{def:mpp} is non-redundant.

\begin{proposition}
    Let \((P,\lambda)\) be a tame marked poset. The description of any generic marked poset polyhedron \(\OO_t(P,\lambda)\) for \(t\in\smash{(0,1)^{\tilde P}}\) given in \Cref{def:mpp} is non-redundant.
\end{proposition}

\begin{proof}
    The way we will prove non-redundance of the description in \Cref{def:mpp} is to reconsider the proof of \Cref{prop:mcop}.
    We have seen that picking a parameter \(u=\chi_C\in\smash{\{0,1\}^{\tilde P}}\) for a partition \(\tilde P=C\sqcup O\) we obtain the description in \Cref{prop:mcop} but there might be multiple chains as in \Cref{def:mpp} such that \eqref{eq:chain} degenerates to the same inequality listed in \Cref{prop:mcop}.
    Since we know the description in \Cref{prop:mcop} is non-redundant for tame marked posets, we can do the following: take a chain \(\Fc\) giving an inequality for \(\OO_t(P,\lambda)\) as in \Cref{def:mpp} and construct a partition \(\tilde P=C\sqcup O\) such that no other chain yields the same inequality as \(\Fc\) for the marked chain-order polyhedron \(\OO_{C,O}(P,\lambda)\).
    Knowing that the description of \(\OO_{C,O}(P,\lambda)\) is non-redundant we conclude that \(\Fc\) can not be omitted in the description of \(\OO_t(P,\lambda)\) either whenever \(u=\chi_C\) is a degeneration of \(t\), in particular when \(t\in\smash{(0,1)^{\tilde P}}\).

    Consider any chain \(\Fc\colon p_0\prec p_1 \prec p_2 \prec \cdots \prec p_r \prec p\) with \(p_0\in P^*\), \(p_i\in \tilde P\) for \(i\ge 1\), \(p\in P\) and \(r\ge 0\).
    Let \(C=\{p_1,\dots,p_r\}\), \(O=\tilde P\setminus (P^*\sqcup C)\) and note that \(p\in P^*\sqcup O\).
    Since \(p_0\in P^*\) and \(p\notin C\), no other chain gives the same inequality in the proof of \Cref{prop:mcop}.
\end{proof}

We finish this section with a discussion of the extended and refined Hibi--Li conjecture. For order and chain polytopes, marked order and chain polytopes as well as admissible marked chain-order polytopes, analogous conjectures were stated in \cite{HL16,Fou16,FF16}.
Let us state the conjecture in full generality here---for possibly unbounded marked chain-order polyhedra with arbitrary partitions \(\tilde P=C\sqcup O\)---and report on what can be said about the conjecture from the above discussion.

\begin{conjecture} \label{conj:general-hibi-li}
    Let \((P,\lambda)\) be a marked poset with all minimal elements marked. Given partitions \(\tilde P=C\sqcup O\) and \(\tilde P=C'\sqcup O'\) such that \(C\subseteq C'\), we have
    \begin{equation*}
        f_i\left( \OO_{C,O}(P,\lambda) \right) \le f_i\left( \OO_{C',O'}(P,\lambda) \right)\quad\text{for all \(i\in\N\).}
    \end{equation*}
\end{conjecture}

This refined version of the conjecture was stated in case of admissible partitions and bounded polyhedra in \cite{FF16}.
It is clear, that it is enough to consider only the case \(C'=C\sqcup\{q\}\) for some \(q\in O\) and by the results of \Cref{sec:poset-transforms} we can assume \((P,\lambda)\) is regular.
When \(q\) is not a chain-order star element, we know that \(\OO_{C,O}(P,\lambda)\) and \(\OO_{C',O'}(P,\lambda)\) are unimodular equivalent by \Cref{prop:unimodular-equiv} and hence their $f$-vectors are identical.
In fact, the statement of \Cref{prop:unimodular-equiv} is a necessary and sufficient condition for unimodular equivalence for tame marked posets and we can count facets to show \Cref{conj:general-hibi-li} holds for tame marked posets in codimension 1:

\begin{proposition}
    Let \((P,\lambda)\) be a tame marked poset and \(\tilde P=C\sqcup O\) any partition. Given \(q\in O\) let \(C'=C\sqcup\{q\}\) and \(O'=O\setminus\{q\}\), then \(\OO_{C,O}(P,\lambda)\) and \(\OO_{C',O'}(P,\lambda)\) are unimodular equivalent if and only if \(q\) is not a chain-order star element.
    Otherwise, the number of facets increases by
    \begin{equation*}
        (k-1)(l-1),
    \end{equation*}
    where \(k\) is the number of saturated chains \(s\prec q_1\prec\cdots\prec q_k\prec q\) with \(s\in P^*\sqcup O\) and all \(q_i\in C\) and \(l\) is the number of saturated chains \(q\prec q_1\prec\cdots\prec q_k\prec s\) with \(s\in P^*\sqcup O\) and all \(q_i\in C\).
\end{proposition}

\begin{proof}
    If \(q\) is not a chain-order star element, the polyhedra are unimodular equivalent by \Cref{prop:unimodular-equiv}.
    For a tame marked poset, the number of facets of \(\OO_{C,O}(P,\lambda)\) is the number of inequalities in \Cref{prop:mcop}, and hence equal to
    \begin{equation*}
        \left| C \right| +
        \left| \left\{\, a\prec p_1\prec\cdots\prec p_r\prec b \,\middle|\,
        r\ge 0, a,b\in P^*\sqcup O, p_i\in C\,\right\} \right|.
    \end{equation*}
    Changing an order element \(q\) to be a chain element, the first summand increases by \(1\), while in the second summand the \(k+l\) chains ending or starting in \(q\) are replaced by the \(kl\) chains now going through \(q\).
    Hence, the number of facets increases by
    \begin{equation*}
        1-(k+l)+kl = (k-1)(l-1).\qedhere
    \end{equation*}
\end{proof}

If \Cref{conj:regular-tame} holds, we can conclude that the Hibi--Li conjecture as formulated in \Cref{conj:general-hibi-li} holds in codimension 1.
For smaller dimensions, we have a common bound on all $f$-vectors of marked chain-order polyhedra associated to a marked poset \((P,\lambda)\) by the $f$-vector of the generic marked poset polyhedron obtained from \Cref{cor:f-vector-degeneration}.
Unfortunately, this does not help for obtaining a comparison as in the Hibi--Li conjecture.

\printbibliography

\end{document}